\documentclass[svjour3]{sn-jnl}%
\usepackage{graphicx}
\usepackage{multirow} 
\usepackage{amsmath,amssymb,amsfonts}%
\usepackage{amsthm}%
\usepackage{mathrsfs}%
\usepackage[title]{appendix}%
\usepackage{soul, color, xcolor}%
\usepackage{textcomp}%
\usepackage{manyfoot}%
\usepackage{booktabs}%
\usepackage{algorithm}
\usepackage{algorithmicx}
\usepackage{algpseudocode}
\usepackage{listings}%
\usepackage{subfigure}

\usepackage{makecell} 
\usepackage{enumitem} 
\soulregister{\cite}7 
\soulregister{\citep}7 
\soulregister{\citet}7 
\soulregister{\ref}7 
\soulregister{\pageref}7 

\allowdisplaybreaks 

\numberwithin{equation}{section}
\numberwithin{table}{section}
\numberwithin{figure}{section}

\newtheorem{theorem}{Theorem}[section]

\newtheorem{lemma}[theorem]{Lemma}
\newtheorem{proposition}[theorem]{Proposition}

\newtheorem{remark}[theorem]{Remark}
\newtheorem{example}[theorem]{Example}
\newtheorem{experiment}[theorem]{Experiment}

\newcommand{\R}{\mathbb{R}}
\newcommand{\E}{\mathbb{E}}

\newcommand{\calh}{\mathcal{H}}
\newcommand{\calj}{\mathcal{J}}
\newcommand{\caln}{\mathcal{N}}
\newcommand{\calo}{\mathcal{O}}
\newcommand{\calx}{\mathcal{X}}
\newcommand{\ph}{\phantom}
\newcommand{\wh}{\widehat}
\newcommand{\wt}{\widetilde}
\newcommand{\eps}{\varepsilon}
\newcommand{\zero}{\mathbf{0}}

\usepackage{xcolor}

\newcommand{\smtxa}[2]{
{\mbox{\scriptsize
$\left[\!\! \begin{array}{#1} #2 \end{array} \!\! \right]$}}}

\makeatletter
\let\c@table\c@figure
\makeatother
\renewcommand{\thetable}{\thefigure}

\begin{document}
\title[Deterministic and randomized Kaczmarz methods for $AXB=C$ with applications to color image restoration]
{Deterministic and randomized Kaczmarz methods for $AXB=C$ with applications to color image restoration}

\author[1]{\sur{Wenli Wang}}\email{wenli.wang12@gmail.com}

\author[1]{\sur{Duo Liu}}\email{20118001@bjtu.edu.cn}

\author[1]{\sur{Gangrong Qu}}\email{grqu@bjtu.edu.cn}

\author*[2]{\sur{Michiel E.~Hochstenbach}}\email{m.e.hochstenbach@tue.nl}

\affil[1]{\orgdiv{School of Mathematics and Statistics}, \orgname{Beijing Jiaotong University}, \orgaddress{\city{Beijing}, \postcode{100044}, \country{China}}}

\affil*[2]{\orgdiv{Department of Mathematics and Computer Science}, \orgname{TU Eindhoven}, \orgaddress{\city{Eindhoven}, \postcode{5600 MB}, \country{The Netherlands}}}

\abstract{
We study Kaczmarz type methods to solve consistent linear matrix equations.
We first present a block Kaczmarz (BK) method that employs a deterministic cyclic row selection strategy.
Assuming that the associated coefficient matrix has full column or row rank, we derive matrix formulas for a cycle of this BK method.
Moreover, we propose a greedy randomized block Kaczmarz (GRBK) method and further extend it to a relaxed variant (RGRBK) and a deterministic counterpart (MWRBK). We establish the convergence properties of the proposed methods.
Numerical tests verify the theoretical findings, and we apply the proposed methods to color image restoration problems.}

\keywords{Deterministic and randomized Kaczmarz methods, consistent linear matrix equation, convergence, minimum norm solution, Kaczmarz for multiple right-hand sides, color image restoration}
\pacs[MSC Classification]{15A06, 15A24, 65F10}

\maketitle

{\small
\begin{center}
{\em Dedicated with love to Prof.~?ke Bj\"{o}rck, for his beautiful character, friendship, \\
and work on many topics, including some appearing in this paper.}
\end{center}}

\section{Introduction}\label{section1}

Given the matrices $A\in \R^{m\times p}$, $B\in \R^{q\times n}$, and $C\in \R^{m\times n}$, we are interested in solving $X\in \R^{p\times q}$ from the consistent matrix equation
\begin{align}\label{eq1.1}
AXB=C.
\end{align}
This equation has close connections to numerous important applications in engineering, such as image processing, stability analysis, control theory, regression analysis, computer vision, image processing, network analysis, and computer-aided geometric design; see, e.g., \cite{ref2,ref3,ref5,ref6,ref11}.
One particular example that we will study is a forward model for noise-free color image restoration, which can be written as
\begin{align}\label{eq1.3}
AXA_{\rm c}^T=C\,;
\end{align}
see, e.g.,~\cite{ref5}.
In \eqref{eq1.1}, $A$ may be square, overdetermined ($m>p$) or underdetermined ($m<p$); and similarly for $B$.
In the context of \eqref{eq1.3}, the coefficient matrix $A$ is typically a large square matrix, whereas $A_{\rm c}$ is of size $3\times 3$.
We will consider all these cases in experiments in Sections~\ref{section5} and \ref{section6}.

``Direct'' approaches for solving \eqref{eq1.1} have cubic complexity in the dimensions, and may not be suitable when the solution is not unique, or if there is noise in $C$.
Various iterative algorithms have been proposed for solving matrix equation~\eqref{eq1.1}~\cite{ref8,ref9,ref10,ref14}.
A common property of these methods is that they typically require storing and accessing the entire coefficient matrix at each iteration.
Instead, we study {\em row-action} methods, which access only one row of $A$ in each iteration.
Row-action and column-action techniques have gained significant attention as efficient iterative algorithms for solving large-scale linear systems~(e.g., \cite{ref17,ref34,ref29,su})
\begin{align}\label{eq1.4}
Ax=b.
\end{align}
These methods decompose the system into subproblems based on row (or sometimes column) partitioning, which may result in several advantages such as memory utilization~\cite{Fe}.
In addition, and perhaps even more importantly, it turns out that row-action methods can often yield high-quality solutions also in the presence of noise or incomplete data (see, e.g., \cite{ref22}).
For our problem \eqref{eq1.1}, because of the presence of the matrix $B$, row-actions may be even more beneficial, for instance since expensive matrix-matrix products are avoided.

The classical Kaczmarz method is a particular row-action method that has been extensively utilized to solve linear system~\eqref{eq1.4}~\cite{ref15}.
At the $k$th iteration, the standard approach is to select the row index $i_k$ in a cyclic manner.
Then the vector in the Kaczmarz method is updated according to
\begin{align}\label{eq1.2}
x^{k+1}=x^k+\alpha \ \tfrac{b_{i_k}-A_{i_k,:}\,x^k}{\|A_{i_k,:}\|^2}\,
A_{i_k,:}^T,
\quad i_k=(k\ \text{mod}\ m)+1.
\end{align}
Here, $A_{i,:}$ denotes the $i$th row of $A$, $b_i$ represents the $i$th component of $b$, $\|\cdot\|$ is the 2-norm, and $\alpha \in (0,2)$ is a relaxation parameter.
The update for $\alpha=1$ is derived by orthogonally projecting the current estimate vector $x^k$ onto the hyperplane $\{x~|~\langle A_{i_k,:}^T\,,\,x\rangle=b_{i_k} \}$.
In the field of medical imaging (Computed Tomography), the Kaczmarz method is also known as the Algebraic Reconstruction Technique~\cite{ref16}.

{\bf Related work.}
Numerous variants have been proposed for the Kaczmarz method.
Strohmer and Vershynin~\cite{ref18}, instead of following the traditional sequential manner, select indices at random with probabilities proportional to the row norms. This results in a randomized Kaczmarz (RK) method with an expected rate of convergence that is called exponential or linear, depending on the community.
Subsequently, Bai and Wu~\cite{ref19} propose a row selection strategy based on weighted residuals.
Specifically, they prioritize the rows with larger residuals, which may lead to faster convergence compared to the method of \cite{ref18}.
Their approach is referred to as the greedy randomized Kaczmarz (GRK) method.
Furthermore, the GRK method can be generalized by inserting a parameter $\theta \in [0,1]$ into the row selection strategy (see Section~\ref{section4.2} for more details), which results in a relaxed greedy randomized Kaczmarz (RGRK) method proposed in \cite{ref20}.
Du and Gao~\cite{ref25} consider the limit case $\theta=1$, which is a deterministic method which directly selects the row with the largest weighted residual. They provide a theoretical estimate of its convergence rate and refer to this approach, originally introduced by McCormick~\cite{Mc}, as the maximal weighted residual Kaczmarz (MWRK) method.
We will extend these GRK, RGRK, and MWRK methods to problem~\eqref{eq1.1}.

We now focus our attention on extensions of Kaczmarz type methods for solving \eqref{eq1.1}.
Wu, Liu, and Zuo \cite{ref11} extend the RGRK and MWRK methods presented in~\cite{ref20,ref25} to solve \eqref{eq1.1} and provide convergence results.
In their methods, the original matrix equation~\eqref{eq1.1} is reformulated into $mn$ subsystems $A_{i,:}XB_{:,j}=C_{i,j}$.
The resulting iteration in each step adds a rank-one matrix. Here, $B_{:,j}$ denotes the $j$th column of $B$ and $C_{i,j}$ is the corresponding element of $C$.
Niu and Zheng \cite{Niu} consider similar methods, in which matrix equation \eqref{eq1.1} is decomposed into a series of subblock problems of the form $A_{I,:} X B_{:,J} = C_{I,J}$, where $A_{I,:}$ and $B_{:,J}$ represent the row and column submatrices of $A$ and $B$ indexed by $I$ and $J$, respectively, and $C_{I,J}$ is the associated submatrix of $C$.
A bit further from our approaches, Du, Ruan, and Sun \cite{Du} consider updating $X$ one column per step.
Our methods, as the approach in \cite{ref11}, also work with rank-one updates, but an essential difference is that the methods in \cite{ref11} consider one element of $C$ and one column of $B$ in each step, while we exploit a row of $C$ and the entire $B$.

Closer to our methods, Xing, Bao, and Li (see \cite{ref12} and Algorithm~\ref{ME-RBK}) decompose \eqref{eq1.1} into $m$ subsystems or row equations $A_{i,:}XB=C_{i,:}$ and extend the RK method in~\cite{ref18} to solve this matrix equation.
To distinguish their method (which uses the entire $B$) from that in reference~\cite{ref11} (which works on columns of $B$), they name it the randomized block Kaczmarz (RBK) method.

\begin{quote}
We would like to emphasize that the use of the term ``block'' in the methods proposed in this paper refers to the presence of $B$, rather than the (also common) meaning of working with the pseudoinverse of a subblock of $A$ (cf.~\cite{Elf, DNRY25}).
However, in fact this is natural, since seen on the level of the linear system $(B^T \otimes A)\,\text{vec}(X) = \text{vec}(C)$ associated with \eqref{eq1.1}, we work with $B^T \otimes A_{i,:}$, which corresponds to a subset (or minibatch) of rows of the operator $B^T \otimes A$.
\end{quote}
For other generalizations of the Kaczmarz method applied to the (well-known but different from \eqref{eq1.1}) Sylvester matrix equation, we refer to \cite{ref30}.

{\bf Some connections with other methods.}
The stochastic gradient descent (SGD) method (see, e.g., \cite{Chen,Nee3}) is an iterative optimization algorithm of the form
\begin{align}\label{SGD}
x^{k+1}=x^k-\alpha \ |I_k|^{-1}\,\sum_{i\in I_k} \nabla f_i(x^k),
\end{align}
where $\alpha$ is a learning rate, $I_k$ is a random subset of data, $\sum_{i\in I_k}\nabla f_i(x^k)$ is the gradient calculated from $I_k$, and $|I_k|$ represents the number of elements of $I_k$.
This method has been widely employed in machine learning \cite{Bo}.
Needell, Srebro, and Ward~\cite{Nee3} note that the RK method presented in~\cite{ref18} can be viewed as the SGD method with a minibatch size of one ($|I_k|=1$).

The Landweber method is a classical gradient approach for solving linear system~\eqref{eq1.4} by minimizing the objective function $\tfrac12\,\|Ax-b\|^2$~\cite{Qu}. This leads to the iterative scheme
\[
x^{k+1}=x^k+\alpha\,A^T\,(b-A\,x^k),
\]
where $A^T\,(b-A\,x^k)$ is the negative gradient of the objective function, and the stepsize $\alpha$ satisfies $0<\alpha<2 \, \|A\|^{-2}$.
The gradient-based iterative (GI) algorithm for solving matrix equation \eqref{eq1.1}, as studied in \cite{ref8}, is given by
\[
X^{k+1}=X^k+\alpha\,A^T \, (C-AX^kB) \, B^T,
\]
where $A^T \, (C-AX^kB) \, B^T$ is the negative gradient of the objective function $\tfrac12 \, \|AXB-C\|_F^2$, and the stepsize satisfies $0<\alpha< 2 \, \|A\|^{-2} \, \|B\|^{-2}$.
The Landweber and GI methods use the entire matrix at each step.
All our methods may also be derived from a gradient of an objective function (see Table~\ref{problems}). Instead of using the entire $A$, our methods follow the iterative scheme with $|I_k|=1$ in \eqref{SGD}, employing deterministic or randomized row selection strategies.

{\bf Our contributions.}
This work extends and analyzes Kaczmarz type methods for solving matrix equation (\ref{eq1.1}), covering deterministic and randomized Kaczmarz approaches.
We first propose and study (deterministic) cyclic row selection variants of the randomized RBK method~\cite{ref12}, which we call block Kaczmarz (BK) methods.
We consider a general BK method (Algorithm~\ref{ME-BK}) and show that this method converges to the solution $A^+CB^++X^0-A^+AX^0BB^+$.
Next, we investigate the BK method under specific constraints on the coefficient matrix $B$: full column rank (Algorithm~\ref{ME-BK1}), and full row rank (Algorithm~\ref{alg.AX=C}), with $B$ nonsingular as a special case.
These extra constraints enable the derivation of the concise and elegant matrix formula for a complete cycle of the BK method.
(The terms {\em cycle} and {\em sweep} are used interchangeably in the BK method to refer to a complete sequential traversal of all rows.)
An overview of the BK methods with different assumptions on $B$ is provided in Table~\ref{problems}.
In the case where the coefficient matrix $B$ has full column rank, a QR decomposition $B = QR$ is used, and the gradient is derived from an associated least-squares function; see Section~\ref{section3} for more details.

\begin{table}[htb!]
\caption{Overview of linear matrix equations and corresponding BK method formulations under various assumptions on $B$.}\label{problems}%
\setlength{\tabcolsep}{3mm}{
\begin{tabular}{lllll}
\hline\rule{0pt}{2.3ex}%
 & General case & $B$ full column rank & $B$ full row rank \\ \hline\rule{0pt}{3ex}%
Matrix equation & $AXB = C$ & $AXQ = \wh C$ & $AX = \wt C$ \\[1mm]
Right-hand side&$C$ & $\wh C = CR^{-1}$ & $\wt C = CB^T\,(BB^T)^{-1}$  \\[1mm]
Gradient form & $A^T (C-AXB)\, B^T$ & $A^T (\wh C-AXQ)\, Q^T$ & $A^T (\wt C-AX)$ \\[1mm]
Kaczmarz step & $A_i^T (C_i-A_iXB)\, B^T$ & $A_i^T (\wh C_i-A_iXQ)\, Q^T$ & $A_i^T (\wt C_i-A_iX)$ \\[1mm]
Sweep formula & -- & $\checkmark$ \ \eqref{eq3.18} & $\checkmark$ \ \eqref{eq0} \\[1mm]
Reference & Sec.~\ref{section3.1} (Alg.~\ref{ME-BK}) & Sec.~\ref{section3.2} (Alg.~\ref{ME-BK1}) & Sec.~\ref{section3.21} (Alg.~\ref{alg.AX=C})\\
\hline
\end{tabular}}
\end{table}

Next, we take the greedy randomized Kaczmarz method \cite{ref19} as a starting point.
This method is designed for linear systems \eqref{eq1.4} and randomly selects a row from those with large weighted residuals.
We extend this method to a greedy randomized block Kaczmarz (GRBK, Algorithm~\ref{ME-GRBK}) method to solve \eqref{eq1.1}.
Furthermore, we propose and study a relaxed version (RGRBK, Algorithm~\ref{ME-RGRBK}) and a maximal weighted residual block Kaczmarz (MWRBK, Algorithm~\ref{ME-MWRBK}) method, which can be regarded as relaxed and deterministic variants of GRBK.
All these approaches are extensions of methods for linear systems \eqref{eq1.4}.
RGRBK includes a tunable parameter $\theta\in[0,1]$. For two special cases of $\theta$, we will see that RGRBK coincides with GRBK when $\theta=\tfrac12$ and with MWRBK when $\theta=1$; see also Table~\ref{index} and Fig.~\ref{deltafig}.
We prove that the GRBK, RGRBK, and MWRBK methods converge in expectation to the unique minimum norm solution $A^+CB^+$ for a zero initial matrix. We also show in Section~\ref{section4} that the upper bounds on the convergence factors of these three methods do not exceed that of the RBK method in \cite{ref12}.

An overview of some methods for solving \eqref{eq1.1} and \eqref{eq1.4}, relevant to this paper are shown in Table~\ref{methods}.
We note that all schemes for \eqref{eq1.1} add rank-one updates in each iteration, with the exception of the GI method.

\begin{table}[h]
\caption{Overview of some relevant methods for solving $Ax=b$ and $AXB=C$.}
\label{methods}%
\setlength{\tabcolsep}{1mm}{
\begin{tabular}{lllll}
\hline\rule{0pt}{2.3ex}%
Row-action & $Ax=b$ & $AXB=C$ & Remark for $AXB=C$ \\
\hline\rule{0pt}{2.8ex}%
Deterministic & Kaczmarz~\cite{ref15}& BK (Alg.~\ref{ME-BK}) & Cyclic row actions \\[0.5mm]
& & BK (Alg.~\ref{ME-BK1}) & $B$ full column rank \\[0.5mm]
& & BK (Alg.~\ref{alg.AX=C}) &  $B$ full row rank \\[0.5mm]
& MWRK~\cite{ref25,Mc}& MWRK~\cite{ref11}, MWRBK (Alg.~\ref{ME-MWRBK}) & ``Max row'', RGRBK for $\theta=1$ \\[2.5mm]
Randomized& RK~\cite{ref18}& RBK (Alg.~\ref{ME-RBK})~\cite{ref12} &  \\[0.5mm]
& GRK~\cite{ref19}& GRBK (Alg.~\ref{ME-GRBK}) & RGRBK for $\theta=\tfrac12$ \\[0.5mm]
& RGRK~\cite{ref20}& RGRK~\cite{ref11}, RGRBK (Alg.~\ref{ME-RGRBK}) &  Involves $\theta \in [0,1]$ \\[2.5mm]
All rows & Landweber~\cite{Qu} & GI~\cite{ref8} & Gradient \\
\hline
    \end{tabular}}
\end{table}

The rest of the paper has been organized as follows. Section~\ref{section2} gives some necessary notations and preliminaries. Section~\ref{section3} shows the BK methods and their convergence analysis. Section~\ref{section4} investigates the GRBK method and its relaxed and deterministic variants, and analyzes their convergence properties. Section~\ref{section5} reports some numerical examples to illustrate the theoretical findings, giving an impression on the performance of the various methods. Section~\ref{section6} treats a special application of the presented methods in color image restoration, followed by some conclusions in Section~\ref{section7}.

\section{Notations and preliminaries}\label{section2}

For a matrix $A$, we use $A^+$, $A^T$, $r(A)$, $\text{range}(A)$, $\|A\|_F$, $A_{i,:}^T$, $\rho(A)$, $\sigma_{\min}(A)$, and $\text{tr}(A)$
to represent the Moore--Penrose pseudoinverse, transpose, rank, range space, Frobenius norm, transpose of the $i$th row, the spectral radius, the smallest singular value, and the trace of $A$, respectively.
The inner product and the Kronecker product of the matrices $A$ and $B$ are denoted as $\langle A,B \rangle$ and $A\otimes B$ respectively.
The stretching operator is defined by $\text{vec}(X)=(x_1^T,\,\ldots,\,x_n^T)^T\in \R^{mn}$ for matrix $X=(x_1,\,\ldots,\,x_n)\in \R^{m\times n}$.
We use ${\bf 0}$ to denote the zero matrix or vector.
For two positive integers $a$ and $b$, $(a\ \text{mod}\ b)$ denotes the remainder produced by dividing $a$ by $b$.
For a positive integer $m$, let $[m]=\{1,\,\ldots,\,m\}$.
The expected value conditional on the first $k$ iterations is denoted as $\E_k$, i.e., $\E_k[\cdot]=\E[\,\cdot\,|\,i_{0},\,i_1,\,\ldots,\,i_{k-1}]$,
where $i_{m}\,(m=0,1,\,\ldots,\,k-1)$ is the row index selected at the $m$th iteration.

We will first review some well-known properties of matrix equation \eqref{eq1.1}; cf., e.g., \cite{ref32,Du}.
Matrix equation~\eqref{eq1.1} is consistent if there exists at least one solution $X\in \R^{p\times q}$ such that the equation holds; in this case, the general solution is
\begin{align} \label{XM}
X_{\text{M}}=A^+CB^++U-A^+AUBB^+,
\end{align}
where $U\in\R^{p\times q}$ is arbitrary.
Define
\begin{align}\label{eq2.1}
X_{\ast}:=A^+CB^+\,;
\end{align}
this is the unique minimum norm solution, i.e., $X_{\ast}=\mathop{\arg\min}\limits_{AXB=C}\|X\|_F$.
This holds because one can verify that $\langle X_{\ast},\, U-A^+AUBB^+\rangle=0$ and therefore $\|X_{\text{M}}\|_F^2
 =\|X_{\ast}\|_F^2+\|U-A^+AUBB^+\|_F^2
\ge\|X_{\ast}\|_F^2$.
Note that $A^+ = (A^T\!A)^{-1}A^T$ when $A$ is of full column rank, and $A^+ = A^T(AA^T)^{-1}$ when it is of full row rank, and similarly for $B$.
It is well known that matrix equation \eqref{eq1.1} can be ``vectorized'' to the form $(B^T \otimes A)\,\text{vec}(X)=\text{vec}(C)$.
A rank analysis of $B^T \otimes A$ for the consistent matrix equation (\ref{eq1.1}) shows that $A^+ C B^+$ is its unique solution only in the case that $A$ is of full column rank and $B$ is of full row rank, and is the unique minimum norm solution in all other cases.

Xing et al. propose a randomized row method for \eqref{eq1.1} in~\cite{ref12}, which they call the RBK method, as outlined in Algorithm~\ref{ME-RBK}.
This will be a starting point for several new variants that we develop in this paper.
There are no restrictions on $A$, except for the fact that zero rows should be omitted from matrix equation \eqref{eq1.1}.
We mention that all methods that we study in this paper work with rows of $A$, with the idea that the size of $A$ is at least as large as that of $B$. If this is not the case, we may consider the transposed equation $B^T X^T A^T = C^T$ instead.

\begin{algorithm}\small
\caption{~The RBK method~\cite{ref12}\label{ME-RBK}}
{\begin{algorithmic}[1]
\State {\bf Input}: initial guess $X^0$
\State {\bf Output:} approximate solution to matrix equation~\eqref{eq1.1}
\State {\bf for} $k=0$, $1$, $2$, $\dots$ until convergence {\bf do}
\State ~~~~Select $i_k\in\{1,\,\ldots,\,m\}$ with probability
$\mathbb{P}({\rm row=}\,i_k)
= \|A_{i_k,:}\|^2 \, / \, \|A\|_F^2$
\State ~~~~Update $X^{k+1}=X^k+\tfrac{\alpha}{\|A_{i_k,:}\|^2}\,A_{i_k,:}^T\,((
C_{i_k,:}-A_{i_k,:}\,X^kB)\,B^T)$
\State {\bf End for}
\end{algorithmic}}
\end{algorithm}

The iterative loop of Algorithm~\ref{ME-RBK} outlines the process of updating the approximate solution $X^k$ for \eqref{eq1.1}.
The iteration continues until the approximation achieves the desired accuracy, as discussed in Sections~\ref{section5} and~\ref{section6}.
The update operation in line~5 of Algorithm~\ref{ME-RBK} has a per-iteration complexity of $\calo(q(p+n))$, adding the product of the column vector $A_{i_k,:}^T$ and the row vector $(C_{i_k,:}-A_{i_k,:}\,X^kB)\,B^T$, i.e., a rank-one update.
Note that, due to the presence of the factor $\|A_{i_k,:}\|^2$ in the denominator in line~5, this method is invariant under scaling the rows of $A$ and $C$ (i.e., it remains the same when applied to $DAXB = DC$ for a nonsingular diagonal matrix $D$).
It has been proven in~\cite[Thm.~1]{ref12} that the RBK method converges in expectation to the minimum norm solution $X_{\ast}$ when \eqref{eq1.1} is consistent, as in the following result.

\begin{theorem}\label{theorem2.2} \rm \cite[Thm.~1]{ref12}
Let matrix equation~\eqref{eq1.1} with $A\in \R^{m\times p}$, $B\in \R^{q\times n}$ and $C\in \R^{m\times n}$ be consistent.
If $0<\alpha< 2 \, \|B\|^{-2}$,
then the sequence $\{X^k\}_{k=0}^{\infty}$, obtained by RBK from an initial matrix $X^0\in \R^{p\times q}$, in which $X_{:,j}^0\in \text{range}(A^T)$, $j=1,\,\ldots,\,q$ and $(X_{i,:}^0)^T\in \text{range}(B)$, $i=1,\,\ldots,\,p$, converges in expectation to the unique minimum norm solution $X_{\ast}$. Moreover, the solution error satisfies
\begin{align*}
\E\big[\|X^k-X_{\ast}\|_F^2\big]
\le \delta^k \cdot \|X^0-X_{\ast}\|_F^2,
\end{align*}
where
\begin{align} \label{delta}
\delta = \delta(\alpha) = 1-\tfrac{2\alpha-\alpha^2\,\|B\|^2}
{\|A\|_F^2}\ \sigma_{\min}^2(A)\ \sigma_{\min}^2(B).
\end{align}
\end{theorem}

\noindent
Note that the common choice of $X^0 = \zero$ satisfies the two range conditions in the previous theorem. This zero matrix is a standard choice for Kaczmarz type methods, and we will also exploit it in Sections~\ref{section5} and~\ref{section6}.
In addition, we will compare the convergence factor $\delta$ in Theorem~\ref{theorem2.2} with those of the algorithms discussed in Section~\ref{section4}. A summary of their convergence factors is provided in Table~\ref{factors}.
With the common notation for the (standard and mixed) condition numbers $\kappa(A) = \|A\| / \sigma_{\min}(A)$ and $\kappa_{F,2}(A) = \|A\|_F / \sigma_{\min}(A)$, we have that $\delta$ is minimal for the choice $\alpha = \|B\|^{-2}$, in which case
\[
\delta = \delta(\|B\|^{-2}) = 1 - \kappa_{F,2}(A)^{-2} \, \kappa(B)^{-2}.
\]
Therefore, this is closer to 1 (and hence less favorable) than the factor $1 - \kappa_{F,2}(A)^{-2}$ for randomized Kaczmarz for linear systems \eqref{eq1.4} \cite{ref18}.

\section{Block Kaczmarz methods and their convergence}\label{section3}

Inspired by the RBK method (Algorithm~\ref{ME-RBK}), we now propose and study a deterministic cyclic variant for solving \eqref{eq1.1}, which we call the block Kaczmarz (BK) method (where we recall the meaning of the word ``block'' as explained in the introduction).
In Section~\ref{section3.1}, we consider a general BK method (Algorithm~\ref{ME-BK}), which imposes no restrictions on the coefficient matrices.
In Sections~\ref{section3.2} and \ref{section3.21}, we investigate the BK method under different constraints on the coefficient matrix $B$:
full column rank (Section~\ref{section3.2}, Algorithm~\ref{ME-BK1}), and full row rank (Section~\ref{section3.21}, Algorithm~\ref{alg.AX=C}).
This results in slightly simplified iteration schemes with matrix formulas for a full sweep of steps.

A word about notation: in the methods, we need to distinguish between a single row iteration, and a complete sweep, i.e., one full cycle through all $m$ rows meaning $m$ row-action steps.
The sweeps are important in the proof of Theorem~\ref{theorem3.3} and in Algorithms~\ref{ME-BK1} and \ref{alg.AX=C}.
All other methods and results use row actions.
With a little overload of notation, the following convention turns out to be convenient:

\begin{quote}
the notation $X^k$ is used to denote the result of single row operations, and $X^s$ (and $X^{s,i}$) for the iterates using sweeps.
\end{quote}

\subsection{The general BK method}\label{section3.1}

In Algorithm~\ref{ME-BK}, we give a pseudocode of the BK method suitable for any size of $B$.

\begin{algorithm}\small
\caption{~The general BK method\label{ME-BK}}
{\begin{algorithmic}[1]
\State {\bf Input}: initial guess $X^0$
\State {\bf Output:} approximate solution to matrix equation~\eqref{eq1.1}
\State {\bf for} $k=0$, $1$, $2$, $\dots$ until convergence {\bf do}
\State ~~~~Compute $i_k=(k\ \text{mod}\ m)+1$
\State ~~~~Compute $X^{k+1}=X^k+\tfrac{\alpha}{\|A_{i_k,:}\|^2}\,A_{i_k,:}^T\,((C_{i_k,:}-A_{i_k,:}\,X^kB)\,B^T)$
\State {\bf end for}
\end{algorithmic}}
\end{algorithm}

The only difference between Algorithms~\ref{ME-RBK} and \ref{ME-BK} is the row selection index $i_k$: Algorithm~\ref{ME-RBK} randomly selects an index based on probability $\|A_{i_k,:}\|^2 \, / \, \|A\|_F^2$, while Algorithm~\ref{ME-BK} sweeps all rows sequentially.
Algorithm~\ref{ME-BK} can also be interpreted as a {\em row-wise} gradient descent scheme with stepsize $\alpha$ for solving a scaled matrix equation
\begin{align}\label{DAXB}
DAXB = DC,
\end{align}
where $D$ is the diagonal matrix $\text{diag}(\|A_{1,:}\|^{-1},\, \dots,\, \|A_{m,:}\|^{-1})$.
For the $i$th row of \eqref{DAXB}, the corresponding least-squares problem is given by
\begin{align*}
\min_XF_i(X) := \min_X \tfrac12 \, \|A_{i,:}\|^{-2} \ \, \|A_{i,:}\,XB-C_{i,:}\|^2.
\end{align*}
It may be checked that the gradient with respect to $X$ is
\begin{align*}
\nabla F_i(X) = \|A_{i,:}\|^{-2} \ A_{i,:}^T \, (A_{i,:}\,XB-C_{i,:}) \, B^T.
\end{align*}
As Algorithm~\ref{ME-RBK}, the computational cost per iteration of Algorithm~\ref{ME-BK} is quadratic: $\calo(q(p+n))$.

Algorithm~\ref{ME-BK} can also be considered as a row version of the GI method of Ding, Liu, and Ding \cite{ref8}.
Unlike the method in \cite{ref8}, Algorithm~\ref{ME-BK} first decomposes matrix equation~\eqref{eq1.1} into $m$ subsystems of the form $A_{i,:}XB=C_{i,:}$ and updates $X^k$ at each iteration using only a single row of $A$ and the corresponding row of $C$.
The method of \cite{ref8} involves matrix-matrix products, and therefore one step is relatively expensive.
One can check that an update of the GI method requires $\min(\calo(mq(p+n)), \, \calo(pn(m+q)))$ work.
(This minimum arises because one can choose to first compute either $AX$ or $XB$.)
This method results in cubic computational complexity with respect to the sizes $m, n, p, q$. By contrast, Algorithm~\ref{ME-BK} has quadratic complexity and is therefore much cheaper per iteration.
The hope is that the $m$ steps of Algorithm~\ref{ME-BK} make more progress than one step of the GI method.
This is indeed confirmed by the experimental results presented in Section~\ref{section6}.

We now study the convergence of Algorithm~\ref{ME-BK}.
In \cite{ref17}, Jiang and Wang study the convergence of a row-action method for solving \eqref{eq1.4}.
We extend their \cite[Thm.~A.3]{ref17} to \eqref{eq1.1}, by proving that Algorithm \ref{ME-BK} converges to $X_{\ast}^0$, given by
\begin{align}\label{X}
X_{\ast}^0 := X_{\ast}+X^0-A^+A\,X^0\,BB^+,
\end{align}
where $X_{\ast}$ is as in~\eqref{eq2.1}.
The following lemma plays a role in the subsequent convergence analysis.

\begin{lemma}\label{lemma3.2}
\rm The iteration matrix $X^k$ obtained by Algorithm~\ref{ME-BK} from an arbitrary initial guess $X^0$ satisfies
\begin{align*}
X^k-A^+A\,X^kBB^+=X^0-A^+A\,X^0BB^+,\quad k=0,\,1,\,2,\,\ldots\,.
\end{align*}
\end{lemma}
\noindent {\bf Proof.}
From Algorithm~\ref{ME-BK}, we get
\begin{align}\label{eq3.2}
X^k=X^0+\sum_{j=0}^{k-1}
\tfrac{\alpha}{\|A_{i_j,:}\|^2}\,A_{i_j,:}^T\,(
C_{i_j,:}-A_{i_j,:}\,X^jB)\,B^T.
\end{align}
Subsequently, it is obtained that
\begin{align}
A^+A\,X^kBB^+& =A^+A\,\Big[X^0+\sum_{j=0}^{k-1}
\tfrac{\alpha}{\|A_{i_j,:}\|^2}\,A_{i_j,:}^T\,(
C_{i_j,:}-A_{i_j,:}\,X^jB)\,B^T\Big]BB^+ \nonumber \\
& =A^+A\, X^0BB^++\sum_{j=0}^{k-1}
\tfrac{\alpha}{\|A_{i_j,:}\|^2}\,A_{i_j,:}^T\,(
C_{i_j,:}-A_{i_j,:}\,X^jB)\,B^T
. \label{eq3.3}
\end{align}
Let $e_i$ be the $i$th column of the identity matrix.
The last step is valid because
$A^+A\,A_{i_s,:}^T
=(A^+A)^T\,A^T\,e_{i_s}
=A^T\,e_{i_s}
=A_{i_s,:}^T$
and
$B^TB\,B^+=B^T\,(BB^+)^T=B^T\,(B^T)^+\,B^T=B^T$.
By \eqref{eq3.2} and \eqref{eq3.3}, the conclusion is proved. ~~~~~~\fbox {}

\begin{theorem}\label{theorem3.3} \rm (Extension of \cite[Thm.~A.3]{ref17})
Let matrix equation \eqref{eq1.1} with $A\in \R^{m\times p}$, $B\in \R^{q\times n}$ and $C\in \R^{m\times n}$
be consistent.
If $0<\alpha<\tfrac{2}{\|B\|^2}$,
then the sequence $\{X^k\}_{k=0}^{\infty}$, obtained by Algorithm~\ref{ME-BK} from an arbitrary initial matrix $X^0\in \R^{p\times q}$, converges to the solution $X_{\ast}^0$.
\end{theorem}
\noindent {\bf Proof.}
By Algorithm~\ref{ME-BK}, one has
\begin{align}
\|X^{k+1}-X_{\ast}\|_F^2& =\|X^{k+1}-X^k+X^k-X_{\ast}\|_F^2 \nonumber \\
& =\|X^{k+1}-X^k\|_F^2
+2\,\langle X^{k+1}-X^k\,,\, X^k-X_{\ast}\rangle
+\|X^k-X_{\ast}\|_F^2 \nonumber \\
& =\tfrac{\alpha^2}{\|A_{i_k,:}\|^{4}}
\,\|A_{i_k,:}^T\,(
C_{i_k,:}-A_{i_k,:}\,X^kB) \, B^T\|_F^2 \nonumber \\
&  \ph{MMM} +2\,\tfrac{\alpha}{\|A_{i_k,:}\|^2}
\,\big\langle A_{i_k,:}^T\,(
C_{i_k,:}-A_{i_k,:}\,X^kB) \, B^T\, , \,X^k-X_{\ast}\big\rangle \nonumber \\
& \ph{MMM}
+\|X^k-X_{\ast}\|_F^2
. \label{eq3.4}
\end{align}
Note that (making use of the Frobenius norm of a rank-one matrix)
\begin{align*}
\tfrac{\alpha^2}{\|A_{i_k,:}\|^{4}}
\,\|A_{i_k,:}^T\,
(C_{i_k,:}-A_{i_k,:}\,X^kB) \, B^T\|_F^2
& =\tfrac{\alpha^2}{\|A_{i_k,:}\|^2}\,\|
(C_{i_k,:}-A_{i_k,:}\,X^kB) \, B^T\|^2 \\
& \le \tfrac{\alpha^2\,\|B\|^2}{\|A_{i_k,:}\|^2}
\,\|
C_{i_k,:}-A_{i_k,:}\,X^kB\|^2.
\end{align*}
Using $AX_{\ast}B=C$, one has
\begin{align*}
&\big\langle A_{i_k,:}^T\,(C_{i_k,:}-A_{i_k,:}\,X^kB) \, B^T\, ,\,
X^k-X_{\ast}\big\rangle \\
& =\text{tr}\,[(C_{i_k,:}-A_{i_k,:}\,X^kB)^TA_{i_k,:}\,(X^k-X_{\ast})B] \\
& =\text{tr}\,[(C_{i_k,:}-A_{i_k,:}\,X^kB)^T(A_{i_k,:}\,X^kB -C_{i_k,:})] = -\|
C_{i_k,:}-A_{i_k,:}\,X^kB\|^2.
\end{align*}
Substituting the results of the previous two math displays into (\ref{eq3.4}) yields
\begin{align}
&\|X^{k+1}-X_{\ast}\|_F^2-\|X^k-X_{\ast}\|_F^2 \nonumber \\
& \le \tfrac{\alpha^2\,\|B\|^2}{\|A_{i_k,:}\|^2}\,
\|C_{i_k,:}-A_{i_k,:}\,X^kB\|^2
-\tfrac{2\alpha}{\|A_{i_k,:}\|^2}\,
\|C_{i_k,:}-A_{i_k,:}\,X^kB\|^2 \nonumber \\
& =-\tfrac{2\alpha-\alpha^2\,\|B\|^2}{\|A_{i_k,:}\|^2} \
\|C_{i_k,:}-A_{i_k,:}\,X^kB\|^2
. \label{eq3.7}
\end{align}
If $0<\alpha<\tfrac{2}{\|B\|^2}$, then $\left\{\|X^k-X_{\ast}\|_F^2\right\}$ is a monotonically decreasing sequence and converges to a nonnegative limit $\beta$.
After $s+1$ cycles, that is, $k = (s+1)\,m$ steps, we have
\begin{align}\label{eq3.8}
\|X^{s+1}-X_{\ast}\|_F^2
\,\le \,\|X^{s}-X_{\ast}\|_F^2
-\sum_{i=1}^{m}\tfrac{2\alpha-\alpha^2\,\|B\|^2}{\|A_{i,:}\|^2}\,
\|C_{i,:}-A_{i,:}\,X^{s,\,i}\,B\|^2,
\end{align}
where $X^{s,\,i}$ denotes the $i$th iterate in the $s$th sweep.
Because of the constraint on $\alpha$, we obtain
\begin{align*}
0\,\le \,\sum_{i=1}^{m}\tfrac{2\alpha-\alpha^2\,\|B\|^2}{\|A_{i,:}\|^2}
\,\|C_{i,:}-A_{i,:}\,X^{s,\,i}\,B\|^2
\,\le \, \|X^{s}-X_{\ast}\|_F^2-
\|X^{s+1}-X_{\ast}\|_F^2
.
\end{align*}
Obviously, $\lim\limits_{s\rightarrow\infty}\left(\|X^{s}-X_{\ast}\|_F^2-
\|X^{s+1}-X_{\ast}\|_F^2\right)
=\beta-\beta=0$.
Therefore, we conclude that
\begin{align*}
\lim\limits_{s\rightarrow\infty}\sum_{i=1}^{m}\tfrac{2\alpha-\alpha^2\,\|B\|^2}{\|A_{i,:}\|^2}
\ \|C_{i,:}-A_{i,:}\,X^{s,\,i}\,B\|^2
= 0.
\end{align*}
This means that
\begin{align}\label{eq3.11}
\lim\limits_{s\rightarrow\infty} (C_{i,:}-A_{i,:}\,X^{s,\,i}B) = \zero, \quad i=1,\,\ldots,\,m.
\end{align}
The error sequence $\{\|X^{s}-X_{\ast}\|_F\}$ is bounded, and therefore $\{X^{s}\}$ is bounded.
Suppose that the subsequence $\{X^{s_{p}},\ p=1,\,2,\,\dots\}$ converges to a limit $\widehat{X}$.
From Algorithm~\ref{ME-BK} and (\ref{eq3.11}), we obtain
\begin{align*}
\lim\limits_{p\rightarrow\infty}
\big(X^{s_{p},\,i}-X^{s_{p},\,i-1}\big)
=\zero, \quad i=2,\,\ldots,\,m.
\end{align*}
This implies that $\widehat{X}$ is the limit of every subsequence $\{X^{s_{p},\,i}\}$ for $i=1,\,\ldots,\, m$.
From (\ref{eq3.11}), we get
\begin{align*}
C_{i,:}=\lim\limits_{p\rightarrow\infty}
A_{i,:}\,X^{s_{p}m,\,i}B = A_{i,:}\,\widehat{X}B, \quad i=1,\,\ldots,\,m,
\end{align*}
which implies that $\widehat{X}$ is a solution of \eqref{eq1.1}.

According to \eqref{eq3.7}, it is obvious that
$\|X^{k+1}-\widehat{X}\|_F\le \|X^k-\widehat{X}\|_F$.
Therefore,
$\{\|X^k-\widehat{X}\|_F\}$ is a monotonically decreasing and bounded sequence.
Hence, we conclude that $\lim\limits_{k\rightarrow\infty}\|X^k-\widehat{X}\|_F$ exists. Because $\lim\limits_{p\rightarrow\infty}\|X^{k_{p}m}-\widehat{X}\|_F=0$, it follows that $\lim\limits_{k\rightarrow\infty}\|X^k-\widehat{X}\|_F=0$, i.e., $\lim\limits_{k\rightarrow\infty}X^k=\widehat{X}$.

Define $\caln:=\{X\in \R^{p\times q}\ |\ AXB=\zero\}$.
Let $Z=\widehat{X}-A^+A\widehat{X}BB^+$.
Because $Z\in\caln$ and
\begin{align*} 
\langle \widehat{X}-Z\,,\,Z\rangle
& =\langle A^+A\widehat{X}BB^+\,,\,\widehat{X}-A^+A\widehat{X}BB^+\rangle\\
& =\langle A^+A\widehat{X}BB^+\,,\,\widehat{X}\rangle
-\langle A^+A\widehat{X}BB^+\,,\,A^+A\widehat{X}BB^+\rangle\\
& =\text{tr}(\widehat{X}^TA^+A\widehat{X}BB^+)
-\text{tr}((BB^+)^T\widehat{X}^T(A^+A)^TA^+A\widehat{X}BB^+)\\
& =\text{tr}(\widehat{X}^TA^+A\widehat{X}BB^+)
-\text{tr}
(\widehat{X}^TA^+A\widehat{X}BB^+)=0,
\end{align*}
we have $\|\widehat{X}\|_F^2=\|\widehat{X}-Z\|_F^2+\|Z\|_F^2\ge\|\widehat{X}-Z\|_F^2$.
Therefore, $\widehat{X}-Z$ is the minimum norm solution, i.e., $\widehat{X}-Z=X_{\ast}$.
According to Lemma \ref{lemma3.2}, we get
\begin{align*}
\widehat{X}-A^+A\widehat{X}BB^+
& =\lim\limits_{k\rightarrow\infty}X^k-A^+AX^kBB^+ \nonumber\\
& =\lim\limits_{k\rightarrow\infty}X^0-A^+AX^0BB^+
=X^0-A^+AX^0BB^+.
\end{align*}
It follows from this that $\widehat{X}=X_{\ast}+X^0-A^+AX^0BB^+=X_{\ast}^0$. ~~~~~~\fbox {}\\

Define
$g(\alpha)=2\alpha-\alpha^2\,\|B\|^2$.
Inequality~(\ref{eq3.7}) reveals that $g(\alpha)$ is maximized over $0< \alpha < 2 \, \|B\|^{-2}$ when $\alpha = \|B\|^{-2}$, which implies the upper bound of the convergence factor is minimized.
In experiments in Sections~\ref{section5} and~\ref{section6}, we therefore adopt this parameter $\alpha=\|B\|^{-2}$.
Note that the optimal $\alpha$ may be slightly differ from the value minimizing the upper bound (see also Figure~\ref{BK_fullcolumnrow}).

We have investigated the BK method in a general setting (with no restrictions on $A$ or $B$) and now turn our attention to the effects of the cases when $B$ is of full column or row rank.

\subsection{The BK method with $B$ of full column rank}\label{section3.2}

If $B$ has full column rank (which implies $q\ge n$), we can express $B$ as $B=QR$, where
$Q\in \R^{q\times n}$ has orthonormal columns and $R\in \R^{n\times n}$ is an upper triangular matrix. This step requires an investment of $\calo(qn^2)$ work; note that $B$ may be of small size in our applications. Then \eqref{eq1.1} becomes
\begin{align}\label{eq3.17}
AXQ=\widehat{C},
\end{align}
where $\widehat{C}=CR^{-1}$.
Some notes on this initial transformation are in order.
It is usually only a sensible idea when $B$ is not of large size, and not nearly rank-deficient, especially in the presence of noise in $C$ (which we do not consider in this paper).
In the same line, we can also consider $AX = CB^+ \ (= C\,(B^T\!B)^{-1}B^T)$, in which case $Q = I$ in (\ref{eq3.17}).
However, for the theoretical results in this section, an orthogonal $Q$ is sufficient; a transformation with $B^+$ is not needed.

The update formula for the BK method corresponding to problem \eqref{eq3.17} is given by
\begin{align}\label{X_AXQ=C}
X^{k+1}=X^k+
\tfrac{\alpha}{\|A_{i_k,:}\|^2}\,A_{i_k,:}^T\,(
\widehat{C}_{i_k,:}-A_{i_k,:}\,X^k\,Q
)\,Q^T.
\end{align}
A pseudocode for a complete cycle of the iterative scheme \eqref{X_AXQ=C} is given in Algorithm~\ref{ME-BK1}.
As stated before, we slightly reformulate the algorithm in terms of sweeps (using $s$ as notation) to derive a formula in Proposition~\ref{proposition3.5}.
As with Algorithm~\ref{ME-BK}, the rank-one update in line~5 requires $\calo(q(p+n))$ operations per iteration.

\begin{algorithm}\small
\caption{~The BK method with $B$ of full column rank\label{ME-BK1}}
{\begin{algorithmic}[1]
\State {\bf Input}: initial guess $X^0$
\State {\bf Output:} approximate solution to matrix equation~\eqref{eq3.17}
\State {\bf for} $s=0$, $1$, $2$, $\dots$ until convergence {\bf do}
\State ~~~~Set
 $X^{s,\,0}=X^s$
\State ~~~~Compute $X^{s,\,i}=X^{s,\,i-1}+
\tfrac{\alpha}{\|A_{i,:}\|^2}\,A_{i,:}^T\,((
\widehat{C}_{i,:}-A_{i,:}\,X^{s,\,i-1}\,Q
)\,Q^T),\quad i=1,\,\ldots,\,m$
\State ~~~~Update $X^{s+1}=X^{s,\,m}$
\State {\bf end for}
\end{algorithmic}}
\end{algorithm}

The orthogonality of the matrix $Q$ enables the derivation of a matrix expression for one sweep of Algorithm~\ref{ME-BK1}. Generalizing \cite[Prop.~4]{ref34} to problem~\eqref{eq3.17}, we obtain the following formula for a full sweep.

\begin{proposition}\label{proposition3.5}
\rm One cycle of Algorithm~\ref{ME-BK1} can be written as
\begin{align}\label{eq3.18}
X^{s+1}=X^s+A^TL_{\alpha}^{-1}\,(\widehat{C}-AX^s\,Q
)\,Q^T,\quad L_{\alpha} = \text{slt}(AA^T) + \alpha^{-1} \, \text{diag}(AA^T),
\end{align}
where $\text{slt}(AA^T)$ and $\text{diag}(AA^T)$ denote the strictly lower triangular part and the diagonal part of the square matrix $AA^T$, respectively.
\end{proposition}
\noindent {\bf Proof.} Let $X_{i}$ represent the $i\times i$ principal submatrix of the matrix $X$.
Let $X_{1:i,\,:}$ denote the first $i$ rows of the matrix $X$. Suppose that Algorithm~\ref{ME-BK1} can be expressed as
\begin{align*}
X^{s,\,i-1}=X^{s,\,0}+(A_{1:(i-1),\,:})^T\,((L_{\alpha})_{i-1})^{-1}\,\big[(
\widehat{C}-AX^{s,\,0}\,Q
)\,Q^T\big]_{1:(i-1),\,:}.
\end{align*}
Then we have
\begin{align*}
X^{s,\,i}& =X^{s,\,i-1}+
\tfrac{\alpha}{\|A_{i,:}\|^2}\,A_{i,:}^T\,(
\widehat{C}_{i,:}-A_{i,:}\,X^{s,\,i-1}\,Q
)\,Q^T\\
& =X^{s,\,0}+(A_{1:(i-1),\,:})^T\,((L_{\alpha})_{i-1})^{-1}\,\big[(
\widehat{C}-AX^{s,\,0}\,Q
)\,Q^T\big]_{1:(i-1),\,:}\\
&\phantom{M}+\tfrac{\alpha}{\|A_{i,:}\|^2}\,A_{i,:}^T
\big(
\widehat{C}_{i,:}-A_{i,:}\,
\big(X^{s,\,0}+(A_{1:(i-1),\,:})^T\,((L_{\alpha})_{i-1})^{-1}\\
&\phantom{MM}\big[(
\widehat{C}-AX^{s,\,0}\,Q
)\,Q^T\big]_{1:(i-1),\,:}\big)
\,Q\big)\,Q^T\\
& =X^{s,\,0}+(A_{1:(i-1),\,:})^T\,((L_{\alpha})_{i-1})^{-1}\,\big[(
\widehat{C}-AX^{s,\,0}\,Q
)\,Q^T\big]_{1:(i-1),\,:}\\
& \phantom{M}
+\tfrac{\alpha}{\|A_{i,:}\|^2}\,A_{i,:}^T\,(
\widehat{C}_{i,:}-A_{i,:}\,X^{s,\,0}
\,Q
)\,Q^T\\
& \phantom{M}-\tfrac{\alpha}{\|A_{i,:}\|^2}\,A_{i,:}^T\,
A_{i,:}\,
(A_{1:(i-1),\,:})^T\,((L_{\alpha})_{i-1})^{-1}\,\big[(
\widehat{C}-AX^{s,\,0}\,Q
)\big]_{1:(i-1),\,:}
\,Q^T\,Q\,Q^T\\
& =X^{s,\,0}+\left[(A_{1:(i-1),\,:})^T~~ A_{i,:}^T\right]
\smtxa{cc}{
((L_{\alpha})_{i-1})^{-1}& \bf{0}\\[2mm]
-\tfrac{\alpha}{\|A_{i,:}\|^2}
\,A_{i,:}\,(A_{1:(i-1),\,:})^T\,((L_{\alpha})_{i-1})^{-1}& \tfrac{\alpha}{\|A_{i,:}\|^2}} \\
& \ \hspace{70mm} \cdot
\smtxa{c}{\big[(
\widehat{C}-AX^{s,\,0}\,Q
)\,Q^T\big]_{1:(i-1),\,:}\\[2mm]
(\widehat{C}_{i,:}-A_{i,:}\,X^{s,\,0}\,Q
)\,Q^T} \\
& =X^{s,\,0}+\left[(A_{1:(i-1),\,:})^T~~ A_{i,:}^T\right]
\smtxa{cc}{
(L_{\alpha})_{i-1}& \bf{0}\\[2mm]
A_{i,:}\,(A_{1:(i-1),\,:})^T& \alpha^{-1}\|A_{i,:}\|^2}\\
& \ \hspace{70mm} \cdot
\smtxa{cc}{\big[(
\widehat{C}-AX^{s,\,0}\,Q
)\,Q^T\big]_{1:(i-1),\,:}\\[2mm]
(\widehat{C}_{i,:}-A_{i,:}\,X^{s,\,0}\,Q
)\,Q^T} \\
& =X^{s,\,0}+(A_{1:i,\,:})^T\,((L_{\alpha})_{i})^{-1}
\big[(
\widehat{C}-AX^{s,\,0}\,Q
)\,Q^T\big]_{1:i,\,:}.
\end{align*}
Therefore, (\ref{eq3.18}) is obtained by mathematical induction. ~~~~~~\fbox {}\\

\noindent
We recall the following general result regarding a linear iterative process.

\begin{lemma}\label{lemma3.6} \rm (See, e.g., \cite[Thm.~4.2.1]{ref35})
For the iterative scheme $x^{s+1}=M\,x^s+g$, where $M\in \R^{n\times n}$ and $x^s\in \R^n$,
a necessary and sufficient condition for convergence for any initial guess $x^0$ is that the spectral radius of $M$ satisfies $\rho(M)<1$.
\end{lemma}

\noindent
By applying Lemma~\ref{lemma3.6}, we have the following result.

\begin{theorem}\label{theorem3.7}
\rm Let matrix equation \eqref{eq1.1} with $B\in \R^{q\times n}$ have full column rank.
If $\alpha \in (0,2)$, then Algorithm~\ref{ME-BK1} is convergent and $\rho(I_{pq}-(QQ^T \otimes A^T\,L_{\alpha}^{-1}\,A)) < 1$, where the operator is a mapping from $\text{range}(B) \otimes \text{range}(A^T)$ on itself. In addition:
\vspace{-2mm}
\begin{itemize}
\item[(a)] If \eqref{eq1.1} is consistent, the process converges to a solution of this matrix equation.
\item[(b)] If $\text{range}(X^0)\subseteq \text{range}(A^T)$ and $\text{range}((X^0)^T)\subseteq \text{range}(B)$, then the method converges to the minimum norm solution $A^+\wh C\,Q^T$ \ ($= A^+ C B^+$).
\end{itemize}
\end{theorem}
\noindent {\bf Proof.}
By Proposition~\ref{proposition3.5}, we obtain
\begin{align*}
\text{vec}(X^{s+1})& =\text{vec}(X^s)+\text{vec}(A^T\,L_{\alpha}^{-1}\,(\widehat{C}-A\,X^s\,Q
)\,Q^T)\\
& =\text{vec}(X^s)-\text{vec}(A^T\,L_{\alpha}^{-1}\,A\,X^s\,Q\,Q^T)+\text{vec}(A^T\,L_{\alpha}^{-1}\,\widehat{C}\,Q^T)\\
& =\text{vec}(X^s)-(Q\,Q^T \otimes A^T\,L_{\alpha}^{-1}\,A)\,\text{vec}(X^s)+\text{vec}(A^T\,L_{\alpha}^{-1}\,\widehat{C}\,Q^T)\\
& =[I_{pq}-(Q\,Q^T \otimes A^T\,L_{\alpha}^{-1}\,A)]\,\text{vec}(X^s)+\text{vec}(A^T\,L_{\alpha}^{-1}\,\widehat{C}\,Q^T),
\end{align*}
According to Theorem~\ref{theorem3.3}, Algorithm~\ref{ME-BK1} converges if $0<\alpha<2$.
Since the method converges (Lemma~\ref{lemma3.6}) we get the assertion of the spectral radius.
Note that $I_{pq}-(QQ^T \otimes A^T\,L_{\alpha}^{-1}\,A)$ will have eigenvalues equal to one if $q > n$ or if $A$ is not of full column rank;
this results in a component in \eqref{XM} by which the solution differs from $A^+CB^+$.
However, it is the restriction of this operator to the space $\text{range}(B) \otimes \text{range}(A^T)$ that has spectral radius less than 1, and this is what is relevant for the convergence of the Kaczmarz method, considering that the updates are in $\text{range}(A^T)$ with transpose in $\text{range}(B)$. ~~~~~~\fbox {}\\

\subsection{The BK method with $B$ of full row rank}\label{section3.21}

If $B$ has full row rank, matrix equation \eqref{eq1.1} can be reformulated as
\begin{align}\label{AX=C}
AX=\wt C,
\end{align}
where $\wt C=CB^T\,(BB^T)^{-1}$.
This transformation may again be considered under the same circumstances as in Section~\ref{section3.2}.
Note that this is a linear system with multiple right-hand sides, and we perform a Kaczmarz type method involving rank-one updates to this problem.
Then the corresponding update formula for the BK method (cf.~line~5 of Algorithm~\ref{ME-BK}) is
\begin{align}\label{BK_AX=C}
X^{k+1} = X^k + \tfrac{\alpha}{\|A_{i_k,:}\|^2} \, A_{i_k,:}^T \, \big( \wt C_{i_k,:} - A_{i_k,:} X^k \big).
\end{align}
When the matrix $B$ has relatively small dimensions, that is, both $q$ and $n$ are relatively small, the computational cost per iteration is slightly lower than that of Algorithm~\ref{ME-BK}, which has a per-iteration complexity of $\calo(q(p+n))$.
In addition, computing $\wt C$ must be performed in advance, introducing an initial computational cost.
In Experiment~\ref{exp6.2}, we discuss some numerical results using update formula~\eqref{BK_AX=C} in the context of model~\eqref{eq1.3}. Similar to Algorithm~\ref{ME-BK1}, a cycle of iterative formula \eqref{BK_AX=C} is presented in Algorithm~\ref{alg.AX=C}.
This method also includes rank-one update in line~5 with a per-iteration computational cost of $\calo(pq)$.

\begin{algorithm}\small
\caption{~The BK method with $B$ of full row rank\label{alg.AX=C}}
{\begin{algorithmic}[1]
\State {\bf Input}: initial guess $X^0$
\State {\bf Output:} approximate solution to matrix equation~\eqref{AX=C}
\State {\bf for} $s=0$, $1$, $2$, $\dots$ until convergence {\bf do}
\State ~~~~Set
 $X^{s,\,0}=X^s$
\State ~~~~Compute $X^{s,\,i}=X^{s,\,i-1}+
\tfrac{\alpha}{\|A_{i,:}\|^2}\,A_{i,:}^T\,(
\wt{C}_{i,:}-A_{i,:}\,X^{s,\,i-1}
),\quad i=1,\,\ldots,\,m$
\State ~~~~Update $X^{s+1}=X^{s,\,m}$
\State {\bf end for}
\end{algorithmic}}
\end{algorithm}

Algorithm~\ref{alg.AX=C} may be viewed as a natural extension of standard Kaczmarz method (\ref{eq1.2}) for multiple right-hand sides.
To the best of our knowledge, this method has not been proposed or studied before in the literature.
Unlike treating $AX = Y$ as a collection of $q$ independent linear systems $AX_{:,j} = Y_{:,j}$ and solving each separately, Algorithm~\ref{alg.AX=C} updates all columns of $X$ simultaneously at each iteration.

Based on the convergence results presented in Section~\ref{section3.2}, we derive the following corollary.
Equation~\eqref{eq0} is a direct generalization for multiple right-hand sides of that of the classical Kaczmarz algorithm \cite[Prop.~4]{ref34}.
Item (a) is an extension of \cite[Thm.~5]{ref34}, while part (b) is unique to the context of multiple right-hand sides.

\begin{theorem}
\label{BK_full_row} \rm
One cycle of Algorithm~\ref{alg.AX=C} can be written as
\begin{align}\label{eq0}
X^{s+1}=X^s+A^TL_{\alpha}^{-1}\,(\wt C-AX^s).
\end{align}
If $\alpha \in (0,2)$, then Algorithm~\ref{alg.AX=C} is convergent and $\rho\,(I_p-A^T\,L_{\alpha}^{-1}\,A)<1$.
In addition:
\begin{itemize}
\vspace{-2mm}
\item[(a)] If \eqref{eq1.1} is consistent, the process converges to a solution of this matrix equation.
\item[(b)] If $\text{range}(X^0)\subseteq \text{range}(A^T)$ and $\text{range}((X^0)^T)\subseteq \text{range}(B)$, then the method converges to the minimum norm solution $A^+\wt C$ \ ($= A^+ C B^+$).
\end{itemize}
\end{theorem}
\noindent {\bf Proof.}
By replacing $Q$ with the identity matrix $I$ in Proposition~\ref{proposition3.5}, \eqref{eq0} for a cycle of Algorithm~\ref{alg.AX=C} can be directly obtained.

Based on \eqref{eq0}, by replacing $Q$ with $I$ in Theorem~\ref{theorem3.7}, we get
\begin{align*}
\text{vec}(X^{s+1}) =[I_{pq}-(I_q \otimes A^T\,L_{\alpha}^{-1}\,A)]\,\text{vec}(X^s)+\text{vec}(A^T\,L_{\alpha}^{-1}\,\widehat{C}\,Q^T),
\end{align*}
where $I_{pq}-(I_q \otimes A^T\,L_{\alpha}^{-1}\,A)$ is a partitioned diagonal matrix of dimension $pq \times pq$, with each diagonal block equal to $I_q- A^T\,L_{\alpha}^{-1}\,A$.
Thus, we have
\begin{align*}
\rho\,[I_{pq}-(I_q \otimes A^T\,L_{\alpha}^{-1}\,A)]
=\max\limits_{i=1, \dots, p} \ |1-\lambda_i(A^T\,L_{\alpha}^{-1}\,A)|
=\rho\,(I_p-A^T\,L_{\alpha}^{-1}\,A).
\end{align*}
By Theorem~\ref{theorem3.3}, Algorithm~\ref{alg.AX=C} converges if $0<\alpha<2$. Thus, it follows from Lemma~\ref{lemma3.6} that $\rho\,(I_p-A^T\,L_{\alpha}^{-1}\,A)<1$.
~~~~~~\fbox {}\\

Algorithm~\ref{ME-BK1} is developed under the assumption that $B$ has full column rank, while Algorithm~\ref{alg.AX=C} assumes that $B$ has full row rank.
When $B$ is square and nonsingular, both conditions hold. We explore the connections between them in this special case, leading to the following proposition.

\begin{proposition}\label{Alg._3_and_4_are_equivalent}
\rm{Algorithms \ref{ME-BK1} and \ref{alg.AX=C} are equivalent when coefficient matrix $B$ is nonsingular.}
\end{proposition}
\noindent {\bf Proof.}
This equivalence arises from the fact that, if $B$ is square, the matrix $Q$ obtained from the QR decomposition of $B$ is square and hence orthogonal.
We have in this case $\wh C = CR^{-1}$ and $\wt C = C R^{-1} Q^T$.
This implies
\begin{align*}
X^{s,\,i}&=X^{s,\,i-1}+
\tfrac{\alpha}{\|A_{i,:}\|^2}\,A_{i,:}^T\,(
\widehat{C}_{i,:}-A_{i,:}\,X^{s,\,i-1}\,Q
)\,Q^T\\
&=X^{s,\,i-1}+
\tfrac{\alpha}{\|A_{i,:}\|^2}\,A_{i,:}^T\,
(C_{i,:}R^{-1}\,Q^T-A_{i,:}\,X^{s,\,i-1})\\
&=X^{s,\,i-1}+
\tfrac{\alpha}{\|A_{i,:}\|^2}\,A_{i,:}^T\,(
\wt{C}_{i,:}-A_{i,:}\,X^{s,\,i-1}
). ~~~~~~\fbox {}
\end{align*}
In experiments in Section~\ref{section6}, we find that indeed the number of iterations is exactly the same, also in finite precision arithmetic.
We note that matrix equations~\eqref{eq3.17} and \eqref{AX=C} are special cases of the more general form~\eqref{eq1.1}. Consequently, the theoretical results derived for matrix equation~\eqref{eq1.1} remain applicable to equations~\eqref{eq3.17} and \eqref{AX=C}; however, the converse does not necessarily hold.

We now provide a schematic table summarizing key properties of the cyclic and randomized Kaczmarz method.
Just as for linear systems $Ax=b$, we have the following upper bounds for cyclic and randomized Kaczmarz for \eqref{eq1.1}, per step and sweep; and invariance under scaling and permutation of rows.
\begin{center}\small
\begin{tabular}{l|cccc} \hline \rule{0pt}{2.3ex}%
Method & Upper bound step & Upper bound sweep & Scaling & Permutation \\ \hline \rule{0pt}{3.2ex}%
Cyclic  & --- & $\rho(I-A^T L^{-1} A)$ & $\checkmark$ & --- \\
Randomized & $\delta^{1/2}$ & $\delta^{m/2}$ & --- & $\checkmark$ \\ \hline
\end{tabular}
\end{center}

Cyclic Kaczmarz does not have a nontrivial upper bound (i.e., smaller than 1) per step, since at some steps there may not be any progress.
An important note is that the spectral radius gives an \textit{asymptotic} factor (so may not hold initially), and that the randomized upper bound holds \textit{in expectation} (so may also not hold, especially in the beginning;
note that $\sqrt{\E_k [\|X^k-X_*\|_F^2]} \ge \E_k[\|X^k-X_*\|_F]$ so $\delta^{1/2}$ is an upper bound in expectation for the progress in $\|X_k-X_*\|_F$ in one row-action step in the randomized method).
Which one of the sweep bounds for cyclic and randomized is smaller is problem dependent.
The convergence of cyclic Kaczmarz may be (much) faster than randomized, but also the other way around, mainly depending on the row ordering.
If we start by randomly permuting the rows of $A$ (and $C$ as a result), which usually gives a row ordering that is not particularly favorable or unfavorable, then cyclic Kaczmarz typically behaves very similar to its randomized counterpart.

\section{GRBK method and its relaxed and deterministic variants}\label{section4}

In this section, we extend the greedy strategy proposed in~\cite{ref19} to solve matrix equation~\eqref{eq1.1}, resulting in the GRBK method, which is introduced in Section~\ref{section4.1}.
In addition, motivated by the relaxed strategy in~\cite{ref20} and the maximal weighted residual strategy in~\cite{ref25,Mc}, we further improve the GRBK method and obtain its relaxed and deterministic variants, i.e., the RGRBK and MWRBK methods, which are presented in Sections~\ref{section4.2} and~\ref{section4.3}, respectively.

\subsection{The GRBK method}\label{section4.1}

Let us for the moment shift focus to linear system \eqref{eq1.4} with associated residual $r^k := b-Ax^k$.
Bai and Wu~\cite{ref19} observe that rows with larger residual norms exert a greater influence on the overall error, and therefore, they propose a GRK method that prioritizes these rows when solving linear system~\eqref{eq1.4}.
The greedy selection strategy in GRK method involves the maximum and average row norms of $r^k$ relative to $A$ (i.e., $\max\limits_i\tfrac{(r_{i,:}^k)^2}{\|A_{i,:}\|^2}$ and $\tfrac{\|r^k\|^2}{\|A\|_F^2}$), where their proportions sum to $1$ and are balanced using a fixed factor of $\tfrac12$.
Extending the greedy strategy in \cite{ref19}, we propose Algorithm~\ref{ME-GRBK} to solve matrix equation~\eqref{eq1.1}.

\begin{algorithm}\small
\caption{~The GRBK method\label{ME-GRBK}}
\begin{algorithmic}[1]
\State {\bf Input}: initial guess $X^0$
\State {\bf Output:} approximate solution to matrix equation~\eqref{eq1.1}
\State {\bf for} $k=0$, $1$, $2$, $\dots$ until convergence {\bf do}
\State ~~~~Compute $R^0=C-AX^0B$
\State ~~~~Compute $\gamma_k = \tfrac12 \, \Big[
\max\limits_{1\le i\le m} \tfrac{\|R_{i,:}^k\|^2}
{\|A_{i,:}\|^2}
+\tfrac{\|R^k\|_F^2}{\|A\|_F^2}\Big]$
\State ~~~~Determine the index set
$\calj_k=\big\{i~|~
\|R_{i,:}^k\|^2 \ge \gamma_k \,
\|A_{i,:}\|^2
\big\}$
\State ~~~~Let $\widetilde{R}^k$ contain the corresponding rows:
$
\widetilde{R}_{i,:}^k=
\left\{
             \begin{array}{ll}
               R_{i,:}^k, & \text{if}~ i\in\calj_k\\
               \zero, & \text{otherwise}
             \end {array}
           \right.
$
\State ~~~~Choose $i_k\in\calj_k$ with probability
$ \mathbb{P}({\rm row=}\ i_k)
= \|\widetilde{R}_{i_k,:}^k\|^2 \, / \, \|\widetilde{R}^k\|_F^2$
\State ~~~~Update $X^{k+1}=X^k+\tfrac{\alpha}{\|A_{i_k,:}\|^2}\,A_{i_k,:}^T\,(R^k_{i_k,:}\,B^T)$
\State ~~~~Compute $R^{k+1}=R^k-\tfrac{\alpha}{\|A_{i_k,:}\|^2}\,(AA_{i_k,:}^T)\,((R^k_{i_k,:}\,B^T) \, B)$
\State {\bf end for}
\end{algorithmic}
\end{algorithm}

The row selection strategy of the GRBK method is based on the residual.
The two terms in line~5 represent the maximum and average row norms of $R^k$ relative to $A$. These terms also play an important role in Algorithms~\ref{ME-RGRBK} and \ref{ME-MWRBK} and in the analysis.
With the second term, we can select the rows that have above average residual norm (relative to the norm of the rows of $A$). Using the first term only, as we will do in Algorithm~\ref{ME-MWRBK}, we can pick the row with maximal residual norm.
In Algorithm~\ref{ME-GRBK}, we use the average of the two terms.
We consider other weights in Algorithm~\ref{ME-RGRBK}.
The update operation in line~9 of Algorithm~\ref{ME-GRBK} involves a rank-one update, with complexity of $\calo(q(p+n))$.
The computation in line~10 has a similar update with $\calo(qn+mp+mn)$ work.
Therefore, the overall computational cost of Algorithm~\ref{ME-GRBK} per iteration is $\calo(pq+qn+mp+mn)$: quadratic in the matrix sizes.

\begin{remark}\label{remark4.1} \rm
From line~6 of Algorithm~\ref{ME-GRBK}, the index set $\calj_k$ can be expressed as
\begin{align}\label{eq4.0}
\calj_k=\big\{i~\big|~
\tfrac{\|R_{i,:}^k\|^2}
{\|A_{i,:}\|^2}\ge
\tfrac12\max\limits_{1\le i\le m} \tfrac{\|R_{i,:}^k\|^2} {\|A_{i,:}\|^2}
+\tfrac12 \tfrac{\|R^k\|_F^2}{\|A\|_F^2}
\big\},
\end{align}
which filters out the rows with small residual norms.
The index set $\calj_k$ is updated each iteration and is nonempty because of the following.
First, we observe that
\begin{align*}
\max\limits_{1\le i\le m} \tfrac{\|R_{i,:}^k\|^2}{\|A_{i,:}\|^2}
\ge\sum_{i=1}^{m}
\tfrac{\|A_{i,:}\|^2}{\|A\|_F^2}\,
\tfrac{\|R_{i,:}^k\|^2}{\|A_{i,:}\|^2}
= \tfrac{\|R^k\|_F^2}{\|A\|_F^2}.
\end{align*}
Then, letting $\tfrac{\|R_{j,:}^k\|^2}
{\|A_{j,:}\|^2}=\max\limits_{1\le i\le m}\tfrac{\|R_{i,:}^k\|^2}
{\|A_{i,:}\|^2}$, we obtain
\begin{align*}
\tfrac{\|R_{j,:}^k\|^2}
{\|A_{j,:}\|^2}
\ge \tfrac12 \max\limits_{1\le i\le m} \tfrac{\|R_{i,:}^k\|^2}
{\|A_{i,:}\|^2}
+ \tfrac12 \tfrac{\|R^k\|_F^2}{\|A\|_F^2}
= \gamma_k.
\end{align*}
This result shows that $j\in\calj_k$ and therefore $\calj_k\ne\emptyset$.
\end{remark}

\noindent
We now analyze the convergence of Algorithm~\ref{ME-GRBK}.
To do so, we extend \cite[Thm.~3.1]{ref19} (which is for linear systems \eqref{eq1.4}) to problem \eqref{eq1.1}.
We first introduce some relevant quantities that will play a central role in several results to follow.
Let us define
\begin{align}\label{omega}
\Omega_k:=\Big\{i~\big|~ \tfrac{\|R_{i,:}^k\|^2}{\|A_{i,:}\|^2}<\tfrac1m\sum_{i\in[m]}\tfrac{\|R_{i,:}^k\|^2}{\|A_{i,:}\|^2} \Big\}.
\end{align}
This means that $\Omega_k$ is exactly the set of row indices for which the residual norm (relative to $A$) is below average.
It is obvious that $\frac1m\sum_{i\in[m]}\frac{\|R_{i,:}^k\|^2}{\|A_{i,:}\|^2}\le\max\limits_{1\le i\le m}\frac{\|R_{i,:}^k\|^2}{\|A_{i,:}\|^2}$.
In what is to follow, we make the very mild assumption that not all ratios $\|R_{i,:}^k\|^2 \, / \, \|A_{i,:}\|^2$ are exactly the same (which holds with probability one).
Therefore, we can write $\frac1m\sum_{i\in[m]}\frac{\|R_{i,:}^k\|^2}{\|A_{i,:}\|^2}=\eps\cdot\max\limits_{1\le i\le m}\frac{\|R_{i,:}^k\|^2}{\|A_{i,:}\|^2}$, with $0<\eps < 1$.
Then, let $\theta \in [0,1]$, and define
\begin{align} \label{varphi}
\varphi_{k,\,\theta} :=  \theta \cdot \tfrac{1}
{\sum_{i\in[m]\backslash\Omega_k}
{\|A_{i,:}\|^2}
+ \eps \, \sum_{i\in\Omega_k}
{\|A_{i,:}\|^2}}
+ (1-\theta)\cdot \tfrac{1}{\|A\|_F^2}.
\end{align}
It is easy to see that $\varphi_{k,\,\theta} \ge \tfrac{1}{\|A\|_F^2}$ for all $\theta \in [0,1]$.
Moreover, under our mild assumption we know that $\varphi_{k,\,\theta} > \tfrac{1}{\|A\|_F^2}$ for all $\theta \in (0,1]$.
Using this $\varphi_{k,\,\theta}$, we next define a quantity that plays a key role as an upper bound on the convergence factor:
\begin{align} \label{delta1}
\delta_{k,\,\theta} := 1-(2\alpha-\alpha^2\,\|B\|^2) \cdot \varphi_{k,\,\theta} \cdot \sigma_{\min}^2(A) \, \sigma_{\min}^2(B).\end{align}
From the lower bound on $\varphi_{k,\,\theta}$, it follows that
\begin{align} \label{delta2}
\delta_{k,\,\theta} < 1-(2\alpha-\alpha^2\,\|B\|^2) \cdot \|A\|_F^{-2} \cdot \sigma_{\min}^2(A) \, \sigma_{\min}^2(B) = \delta,
\end{align}
where $\delta$ is the upper bound on the convergence factor of Algorithm~\ref{ME-RBK} as in \eqref{delta}.
We note that $\delta$
coincides with $\delta_{k,0}$.
In the next theorem, $\varphi_{k,\,1/2}$ and $\delta_{k,\,1/2}$ are relevant; we will need $\varphi_{k,\,\theta}$ and $\delta_{k,\,\theta}$ for other values of $\theta$ later on.

\begin{theorem}\label{theorem4.2} \rm (Extension of \cite[Thm.~3.1]{ref19})
Let matrix equation~\eqref{eq1.1} with $A\in \R^{m\times p}$, $B\in \R^{q\times n}$ and $C\in \R^{m\times n}$
be consistent.
If $0<\alpha<\tfrac{2}{\|B\|^2}$,
then the sequence $\{X^k\}_{k=0}^{\infty}$, obtained by Algorithm~\ref{ME-GRBK} converges in expectation to a solution of \eqref{eq1.1}. If $\text{range}(X^0)\subseteq \text{range}(A^T)$ and $\text{range}((X^0)^T)\subseteq \text{range}(B)$, then Algorithm~\ref{ME-GRBK} converges in expectation to the unique minimum norm solution $X_{\ast}$.
Moreover, the solution error satisfies
\begin{align}\label{eq4.2}
\E_k\big[\|X^{k+1}-X_{\ast}\|_F^2\big]
\le \delta_{k,\,1/2} \cdot \|X^k-X_{\ast}\|_F^2
\le \delta \cdot \|X^k-X_{\ast}\|_F^2.
\end{align}
Consequently, we have
\begin{align}\label{eq4.3}
\E\big[\|X^{k+1}-X_{\ast}\|_F^2\big]
\le \displaystyle\prod_{\ell=0}^k \delta_{\ell,\,1/2} \cdot
\|X^0-X_{\ast}\|_F^2 \le \delta^{k+1} \cdot \|X^0-X_{\ast}\|_F^2.
\end{align}
\end{theorem}
\noindent {\bf Proof.}
Let $\zeta_k=\tfrac12 \, \Big[\frac{1}{\|R^k\|_F^2}
\max\limits_{1\le i\le m} \tfrac{\|R_{i,:}^k\|^2}
{\|A_{i,:}\|^2} + \tfrac{1}{\|A\|_F^2}\Big]$.
Then, according to line~5 of Algorithm~\ref{ME-GRBK}, for all $i \in \calj_k$ we have $\frac{\|R_{i,:}^k\|^2}{\|A_{i,:}\|^2} \ge \gamma_k = \|R^k\|_F^2 \ \zeta_k$.
For $0<\alpha<\tfrac{2}{\|B\|^2}$, we take the expectation of \eqref{eq3.7}, conditional on the first $k$ iterations, which gives:
\begin{align}\label{GRBK:E_k}
&\E_k\big[\|X^{k+1}-X_{\ast}\|_F^2\big] \nonumber\\
& \ph{MM} \le \|X^k- X_{\ast}\|_F^2-(2\alpha-\alpha^2\,\|B\|^2) \,
\sum_{i=1}^{m}\tfrac{\|\widetilde{R}_{i,:}^k\|^2}{\|\widetilde{R}^k\|_F^2}\,
\tfrac{\|C_{i,:}-A_{i,:}\,X^kB\|^2}
{\|A_{i,:}\|^2} \nonumber\\
& \ph{MM} \le \|X^k- X_{\ast}\|_F^2-(2\alpha-\alpha^2\,\|B\|^2)
\sum_{i\in\calj_k}
\tfrac{\|\widetilde{R}_{i,:}^k\|^2}
{\|\widetilde{R}^k\|_F^2}\,
\zeta_k\,
\|R^k\|_F^2 \nonumber\\
& \ph{MM} = \|X^k- X_{\ast}\|_F^2-(2\alpha-\alpha^2\,\|B\|^2)\ \zeta_k\
\|A\,(X^k- X_{\ast})\,B\|_F^2 \nonumber\\
& \ph{MM} =\|X^k- X_{\ast}\|_F^2-(2\alpha-\alpha^2\,\|B\|^2)\ \zeta_k\
\|(B^T\otimes A)\,\text{vec}(X^k- X_{\ast})\|^2,
\end{align}
where $\calj_k$ is as in \eqref{eq4.0}.
Furthermore, we note that $\text{vec}(X^0-X_{*})\in \text{range}(B\otimes A^T)$ and
\begin{align*}
\text{vec}(\tfrac{\alpha}{\|A_{i_k,:}\|^2}\,A_{i_k,:}^T\,((
C_{i_k,:}-A_{i_k,:}\,X^kB)\,B^T))\in \text{range}(B\otimes A^T).
\end{align*}
Thus, we have $\text{vec}(X^k-X_{*})\in \text{range}(B\otimes A^T)$ by induction. Since $\|A^Tx\|^2\ge\sigma^2_{\text{min}}(A)\,\|x\|^2$
holds for any $x\in \text{range}(A)$ (cf., e.g., \cite[Lemma~2.1]{ref11}), it follows that
\begin{align} \label{sigmamin}
\|(B^T \otimes A) \, \text{vec}(X_k-X_*)\|& \ge \sigma_{\min}(B \otimes A^T) \, \|\text{vec}(X_k-X_*)\| \nonumber\\
&= \sigma_{\min}(A) \, \sigma_{\min}(B) \, \|X_k-X_*\|_F.
\end{align}
Thus, inequality \eqref{GRBK:E_k} becomes
\begin{align}\label{eq4.7}
\E_k\big[\|X^{k+1}-X_{\ast}\|_F^2\big]
\le \big(1 -(2\alpha-\alpha^2\,\|B\|^2) \, \zeta_k \,\sigma_{\min}^2(A) \, \sigma_{\min}^2(B)\big)\cdot
\|X^k- X_{\ast}\|_F^2.
\end{align}
In addition, we have
\begin{align}
\zeta_k\ \|A\|_F^2
& =\tfrac12\,\frac{\max\limits_i\frac{\|R_{i,:}^k\|^2}
{\|A_{i,:}\|^2}}
{\frac{\|R^k\|_F^2}
{\|A\|_F^2}}
+\tfrac12 \nonumber \\
&=\tfrac12\,\frac{\max\limits_i\frac{\|R_{i,:}^k\|^2}
{\|A_{i,:}\|^2}}
{\sum_{i\in[m]\backslash\Omega_k}
\frac{\|A_{i,:}\|^2}{\|A\|_F^2}\,
\tfrac{\|R_{i,:}^k\|^2}{\|A_{i,:}\|^2}
+\sum_{i\in\Omega_k}
\frac{\|A_{i,:}\|^2}{\|A\|_F^2}\,
\tfrac{\|R_{i,:}^k\|^2}{\|A_{i,:}\|^2}}
+\tfrac12 \nonumber \\
&
\ge\tfrac12 \, \frac{1}
{\sum_{i\in[m]\backslash\Omega_k}
\tfrac{\|A_{i,:}\|^2}{\|A\|_F^2}
+ \eps \, \sum_{i\in\Omega_k}
\frac{\|A_{i,:}\|^2}{\|A\|_F^2}}
+\tfrac12. \label{eq:zetaA}
\end{align}
This means that $\zeta_k \ge \varphi_{k,\,1/2}$, which is larger than $\|A\|_F^{-2}$ under the mild assumption that not all residual rows have equal weight.
Applying the law of total expectation $\E[\cdot]=\E[\E_k[\cdot]]$ (cf., e.g.,~\cite{ref19,ref20,ref11}), inequality \eqref{eq4.7} becomes
\begin{align}\label{eq4.12}
\E\big[\|X^{k+1}-X_{\ast}\|_F^2\big]
\le \delta_{k,\,1/2}
\cdot\E\big[\|X^k-X_{\ast}\|_F^2\big].
\end{align}
An induction on the iteration number and \eqref{delta2} gives the estimate (\ref{eq4.3}).
We note that if $\text{range}(X^0)\subseteq \text{range}(A^T)$ and $\text{range}((X^0)^T)\subseteq \text{range}(B)$, the rank-one updates in Algorithm~\ref{ME-GRBK} are such that for all $k$ we keep the property $\text{range}(X^k)\subseteq \text{range}(A^T)$ and $\text{range}((X^k)^T)\subseteq \text{range}(B)$.
Thus, the sequence $\{X^k\}_{k=0}^{\infty}$, converges in expectation to the unique minimum norm solution $X_{\ast}$.
~~~~~~\fbox {}\\

As a side note, in \cite{ref19}, Bai and Wu find that $r^k_{i_{k-1}}=0$, where $r^k_{i_{k-1}}$ is the $i_{k-1}$th component of the associated residual $r^k$. By utilizing this relation, they obtain
\begin{align*}
\frac{\max\limits_i\frac{\|r_{i}^k\|^2}{\|A_{i,:}\|^2}}
{\sum_{i\in[m]}
\frac{\|A_{i,:}\|^2}{\|A\|_F^2}\,
\tfrac{\|r_{i}^k\|^2}{\|A_{i,:}\|^2}}
&=\frac{\max\limits_i\frac{\|r_{i}^k\|^2}{\|A_{i,:}\|^2}}
{\sum_{i\ne i_{k-1}}
\frac{\|A_{i,:}\|^2}{\|A\|_F^2}\,
\tfrac{\|r_{i}^k\|^2}{\|A_{i,:}\|^2}}\\
&
\ge \frac{\|A\|_F^2}
{\sum_{i\ne i_{k-1}}
\|A_{i,:}\|^2}
\ge \frac{\|A\|_F^2}
{\max\limits_{1\le j\le m}\sum_{i\ne j}
\|A_{i,:}\|^2}.
\end{align*}
However, for our matrix equation \eqref{eq1.1}, there will generally be no zero residual row (no index $i$ with $R^k_{i,:}={\bf 0}$). Therefore, we need a different type of set $\Omega_k$ (compared to \cite{ref19}) in the derivation of \eqref{eq:zetaA}.
It would also be possible to exploit our set $\Omega_k$ for the context of linear systems \eqref{eq1.4} as treated in \cite{ref19}.

From inequality (\ref{eq4.12}), the sequence $\big\{\E[\|X^k-X_{\ast}\|_F^2]\big\}$ is decreasing if $0<\alpha<2 \, \|B\|^{-2}$.
Moreover, similar to (\ref{eq3.7}), inequality (\ref{eq4.12}) also includes the function $g(\alpha)=2\alpha-\alpha^2\,\|B\|^2$, which attains its maximum at $\alpha=\|B\|^{-2}$. This indicates that the upper bound of the convergence factor is minimized when the parameter $\alpha$ is set to $\|B\|^{-2}$.

\subsection{The RGRBK method}\label{section4.2}

In the previous subsection, the GRK method employed fixed weights equal to $\tfrac12$ for the maximum and average row norms of $r^k$ relative to $A$.
A natural question arises: what happens if different proportions are used?
To explore this, Bai and Wu~\cite{ref20}, in the context of \eqref{eq1.4}, introduce a variable parameter $\theta\in[0,1]$ and propose an RGRK method.

Two cases of the parameter $\theta$ in the RGRK method are of particular interest:
\begin{itemize}
\vspace{-3.5mm}
\item [$\bullet$] $\theta=\tfrac12$: The RGRK method coincides with the GRK method in~\cite{ref19}.
\item [$\bullet$] $\theta=1$: The RGRK method is equivalent to the MWRK method in~\cite{ref25,Mc}, a {\em deterministic} approach that selects the row with the maximum residual; see also Section~\ref{section4.3}.
\end{itemize}

\noindent
Extending RGRK \cite{ref20} (which is for linear systems \eqref{eq1.4}), we now propose and study an RGRBK method to solve matrix equation~\eqref{eq1.1} in Algorithm~\ref{ME-RGRBK} (cf.~also Table~\ref{methods}).

\begin{algorithm}\small
\caption{~The RGRBK method\label{ME-RGRBK}}
{\begin{algorithmic}[1]
\State {\bf Input}: initial guess $X^0$, parameter $\theta\in[0,1]$
\State {\bf Output:} approximate solution to matrix equation~\eqref{eq1.1}
\State {\bf for} $k=0$, $1$, $2$, $\dots$ until convergence {\bf do}
\State ~~~~Compute $R^0=C-AX^0B$
\State ~~~~Compute $\xi_k=
\theta \cdot
\max\limits_{1\le i\le m}\tfrac{\|R_{i,:}^k\|^2}
{\|A_{i,:}\|^2}
+(1-\theta) \cdot \tfrac{\|R^k\|_F^2}{\|A\|_F^2}$
\State ~~~~Determine the index set
$\calh_k=\big\{i~|~
\|R_{i,:}^k\|^2\,\ge\,\xi_k\,
\|A_{i,:}\|^2
\big\}$
\State ~~~~Let $\widetilde{R}^k$ be the correspondingly selected rows:
$
\widetilde{R}_{i,:}^k=
\left\{
             \begin{array}{ll}
               R_{i,:}^k, & \text{if}~ i\in\calh_k\\
               \zero, & \text{otherwise}\\
             \end {array}
           \right.
$
\State ~~~~Choose $i_k\in\calh_k$ with probability
$ \mathbb{P}({\rm row=}\ i_k)
= \|\widetilde{R}_{i_k,:}^k\|^2 \, / \, \|\widetilde{R}^k\|_F^2$
\State ~~~~Update $X^{k+1}=X^k+\tfrac{\alpha}{\|A_{i_k,:}\|^2}\,A_{i_k,:}^T\,(R^k_{i_k,:}\,B^T)$
\State ~~~~Compute $R^{k+1}=R^k-\tfrac{\alpha}{\|A_{i_k,:}\|^2}\,(AA_{i_k,:}^T)\,((R^k_{i_k,:}\,B^T) \, B)$
\State {\bf end for}
\end{algorithmic}}
\end{algorithm}

Note that Algorithm~\ref{ME-RGRBK} retains the same per-iteration computational complexity as Algorithm~\ref{ME-GRBK}. The reason for the non-emptiness of the index set $\calh_k$ is analogous to Remark~\ref{remark4.1}.
In a manner analogous to Theorem \ref{theorem4.2}, we derive Theorem \ref{theorem4.6}; its proof is omitted.

\begin{theorem}\label{theorem4.6}
\rm{Let matrix equation \eqref{eq1.1} with $A\in \R^{m\times p}$, $B\in \R^{q\times n}$ and $C\in \R^{m\times n}$
be consistent.
If $0<\alpha<\tfrac{2}{\|B\|^2}$,
then the sequence $\{X^k\}_{k=0}^{\infty}$, obtained by Algorithm~\ref{ME-RGRBK} converges in expectation to a solution of \eqref{eq1.1}. If $\text{range}(X^0)\subseteq \text{range}(A^T)$ and $\text{range}((X^0)^T)\subseteq \text{range}(B)$, then Algorithm~\ref{ME-RGRBK} converges in expectation to the unique minimum norm solution $X_{\ast}$.
Moreover, the solution error satisfies
\begin{align*}
\E_k\big[\|X^{k+1}-X_{\ast}\|_F^2\big]
\le \delta_{k,\,\theta}
\cdot\|X^k-X_{\ast}\|_F^2 \le \delta \cdot\|X^k-X_{\ast}\|_F^2,
\end{align*}
As a result, we have
\begin{align*}
\E\big[\|X^{k+1}-X_{\ast}\|_F^2\big]
\le \displaystyle\prod_{\ell=0}^k\delta_{\ell,\theta}\cdot
\|X^0-X_{\ast}\|_F^2 \le \delta^{k+1} \cdot \|X^0-X_{\ast}\|_F^2.
\end{align*}
}
\end{theorem}

\noindent
We will compare the upper bounds of the convergence factors for the algorithms discussed in this section and the RBK algorithm proposed in \cite{ref12} in Section~\ref{section4.3}.

\subsection{The MWRBK method}\label{section4.3}

We now generalize the maximal weighted residual strategy in~\cite{ref25,Mc} (which is for linear systems \eqref{eq1.4}) to matrix equation \eqref{eq1.1} (cf.~also Table~\ref{methods}).
We directly select the row with the largest weighted residual for each iteration, which can also be obtained from the RGRBK method by taking $\theta=1$.
This gives the following MWRBK method (Algorithm~\ref{ME-MWRBK}), which is a {\em deterministic} approach.

\begin{algorithm}\small
\caption{~The MWRBK method\label{ME-MWRBK}}
{\begin{algorithmic}[1]
\State {\bf Input}: initial guess $X^0$
\State {\bf Output:} approximate solution to matrix equation~\eqref{eq1.1}
\State {\bf for} $k=0$, $1$, $2$, $\dots$ until convergence {\bf do}
\State ~~~~Compute $R^0=C-AX^0B$
\State ~~~~Let $i_k$ be the index maximizing $\|R_{i,:}^k\|^2
\,/\,\|A_{i,:}\|^2$; if there are more, choose the smallest
\State ~~~~Update $X^{k+1}=X^k+\tfrac{\alpha}{\|A_{i_k,:}\|^2}\,A_{i_k,:}^T\,(R^k_{i_k,:}\,B^T)$
\State ~~~~Compute $R^{k+1}=R^k-\tfrac{\alpha}{\|A_{i_k,:}\|^2}\,(AA_{i_k,:}^T)\,((R^k_{i_k,:}\,B^T) \, B)$
\State {\bf end for}
\end{algorithmic}}
\end{algorithm}

A theoretical estimate of the convergence rate of Algorithm~\ref{ME-MWRBK} is given below, which is along the same lines as in \cite[Thm.~3.1]{ref25}.

\begin{theorem}\label{theorem4.10} \rm
Let matrix equation \eqref{eq1.1} with $A\in \R^{m\times p}$, $B\in \R^{q\times n}$ and $C\in \R^{m\times n}$
be consistent.
Let $0<\alpha<\tfrac{2}{\|B\|^2}$ and $X^k$ denote the $k$th iteration obtained by Algorithm~\ref{ME-MWRBK} with $\text{range}(X^0)\subseteq \text{range}(A^T)$ and $\text{range}((X^0)^T)\subseteq \text{range}(B)$. Then
\begin{align}\label{eq4.16}
\|X^{k+1}-X_{\ast}\|_F^2
\le \displaystyle\prod_{\ell=0}^k\,\delta_{\ell,\,1} \cdot
\|X^0-X_{\ast}\|_F^2 \le \delta^{k+1} \cdot \|X^0-X_{\ast}\|_F^2,
\end{align}
where $\delta_{\ell,\,1}$ is defined in \eqref{delta1}.
\end{theorem}
\noindent {\bf Proof.}
The method for $\theta=1$ is deterministic, and therefore the proof technique is different from Theorems~\ref{theorem4.2} and \ref{theorem4.6}.
For $0<\alpha<\tfrac{2}{\|B\|^2}$, using that $i_k$ maximizes the quantity in line~5 of Algorithm~\ref{ME-MWRBK}, we have
\begin{align*}
&\|X^{k+1}-X_{\ast}\|_F^2\\
& \le \|X^k- X_{\ast}\|_F^2 -\tfrac{2\alpha-\alpha^2\,\|B\|^2}{\|A_{i_k,:}\|^2}
\,\|R_{i_k,:}^k\|^2 \\
&=\|X^k- X_{\ast}\|_F^2 -(2\alpha-\alpha^2\,\|B\|^2)\,
\tfrac{\|R_{i_k,:}^k\|^2}
{\|A_{i_k,:}\|^2}
\tfrac{\|R^k\|_F^2}
{\sum_{i\in[m]\backslash\Omega_k}\frac{\|R_{i,:}^k\|^2}{\|A_{i,:}\|^2}\,\|A_{i,:}\|^2
+\sum_{i\in\Omega_k}\frac{\|R_{i,:}^k\|^2}{\|A_{i,:}\|^2}\,\|A_{i,:}\|^2} \\
& \le \|X^k- X_{\ast}\|_F^2 -(2\alpha-\alpha^2\,\|B\|^2)\,
\tfrac{\|R_{i_k,:}^k\|^2}
{\|A_{i_k,:}\|^2}
\tfrac{\|R^k\|_F^2}
{\frac{\|R_{i_k,:}^k\|^2}{\|A_{i_k,:}\|^2}\,\big(\sum_{i\in[m]\backslash\Omega_k}\|A_{i,:}\|^2
+\sum_{i\in\Omega_k}\eps\,\|A_{i,:}\|^2
\big)}  \\
&= \|X^k- X_{\ast}\|_F^2 -(2\alpha-\alpha^2\,\|B\|^2)\,
\tfrac{\|A \, (X^k- X_{\ast}) \, B\|_F^2}
{\sum_{i\in[m]\backslash\Omega_k}\|A_{i,:}\|^2+
\eps \, \sum_{i\in\Omega_k}\|A_{i,:}\|^2} \\
& \le \big(1-(2\alpha-\alpha^2\,\|B\|^2) \cdot
\varphi_{k,\,1}\cdot\sigma_{\min}^2(A)\,\sigma_{\min}^2(B)\big) \cdot
\|X^k-X_{\ast}\|_F^2.
\end{align*}
In the last line, we have used \eqref{sigmamin}.
Induction on the iteration number gives inequality~(\ref{eq4.16}).
~~~~~~\fbox {}\\

Similar to the discussions in Sections~\ref{section3.1} and \ref{section4.1}, the upper bound of the convergence factor for the RGRBK and MWRBK methods is minimized when $\alpha=\|B\|^{-2}$.

Table~\ref{factors} provides a summary of the upper bounds of the convergence factors of the methods involved in this section for solving matrix equation~\eqref{eq1.1}.
Recall that $\delta$ is equal to $\delta_{k,0}$.

\begin{table}[htb!]
\caption{The upper bound on the convergence factors of RBK~\cite{ref12}, GRBK, RGRBK, and MWRBK methods for solving $AXB=C$.}
\label{factors}
\setlength{\tabcolsep}{1.1mm}{
\begin{tabular}{lllllll}
\hline\rule{0pt}{2.3ex}%
Method& \multicolumn{2}{l}{Convergence factor upper bound} && Result \\
\hline\rule{0pt}{3.8ex}%
RBK~\cite{ref12}& $\delta = \delta_{k,\,0}$ & \hspace{-2mm} $=1-\frac{2\alpha-\alpha^2\,\|B\|^2}
{\|A\|_F^2}\,\sigma_{\min}^2(A)\,\sigma_{\min}^2(B)$ &\hspace{-2mm} & \cite[Thm.~1]{ref12}  \\ \rule{0pt}{3.8ex}%
GRBK (Alg.~\ref{ME-GRBK})&$\delta_{k,\,1/2}$ & \hspace{-2mm} $=1-(2\alpha-\alpha^2\,\|B\|^2) \cdot \varphi_{k,\,1/2} \cdot
\sigma_{\min}^2(A) \, \sigma_{\min}^2(B)$ & \hspace{-2mm} $< \delta$ &Thm.~\ref{theorem4.2}\\ \rule{0pt}{3.8ex}%
RGRBK (Alg.~\ref{ME-RGRBK})& $\delta_{k,\,\theta}$ & \hspace{-2mm} $=1-(2\alpha-\alpha^2\,\|B\|^2) \cdot \varphi_{k,\,\theta} \cdot
\sigma_{\min}^2(A) \, \sigma_{\min}^2(B)$ & \hspace{-2mm} $< \delta$ &Thm.~\ref{theorem4.6}\\ \rule{0pt}{3.8ex}%
MWRBK (Alg.~\ref{ME-MWRBK})&$\delta_{k,\, 1}$ & \hspace{-2mm} $=1-(2\alpha-\alpha^2\,\|B\|^2) \cdot \varphi_{k,\,1} \cdot \sigma_{\min}^2(A)\,\sigma_{\min}^2(B)$ & \hspace{-2mm} $< \delta$ &Thm.~\ref{theorem4.10}
\\[0.8ex]
\hline
    \end{tabular}}
\end{table}

With respect to these upper bounds, we have already seen that the upper bounds of the convergence factors for GRBK, RGRBK, and MWRBK are not larger than the $\delta$ for RBK.
For linear systems \eqref{eq1.4}, the authors of \cite[Sec.~2]{ref20} show that the upper bound of the convergence factor of RGRK is a monotonically decreasing function with respect to the parameter $\theta$.
We generalize this result for \eqref{eq1.1} by establishing that the upper bound of the convergence factor for the proposed RGRBK method also exhibits a decreasing behavior with respect to $\theta$ (under the very mild assumption of not having equal residual rows as mentioned before):
\[
\frac{\partial\delta_{k,\,\theta}}{\partial\theta}
=(2\alpha-\alpha^2\,\|B\|^2) \, \sigma_{\min}^2(A)\,\sigma_{\min}^2(B)
\cdot \big(\tfrac{1}{\|A\|_F^2}-\tfrac{1}
{\sum_{i\in[m]\backslash\Omega_k}
{\|A_{i,:}\|^2}
+ \eps \, \sum_{i\in\Omega_k}
{\|A_{i,:}\|^2}} \big) < 0.
\]
This means that the upper bound $\delta_{k,\,\theta}$ on the convergence factor is decreasing as a function of $\theta$.
Therefore, we expect that the (deterministic) MWRBK method (corresponding to $\theta=1$) may converge fastest; this is indeed confirmed by experiments in the following sections.

Table~\ref{index} provides an overview of the index selection criteria of the methods involved in this section for solving equations \eqref{eq1.1} and \eqref{eq1.4}.
The RGRBK method has a parameter $\theta \in [0,1]$; it becomes the GRBK method for $\theta=\tfrac12$ and the MWRBK method for $\theta=1$.
Fig.~\ref{deltafig} gives a schematic graph of $\delta_{k,\,\theta}$ as a function of the parameter $\theta$, together with the relations between the RGRBK, GRBK, and MWRBK methods.

\begin{table}[htb!]
\caption{Index selection criteria of the methods involved in this section for solving $Ax=b$ and $AXB=C$.}
\label{index}
\setlength{\tabcolsep}{2.8mm}{
\begin{tabular}{lllll}
\hline\rule{0pt}{3.8ex}%
$Ax=b$ & GRK~\cite{ref19} &  $
\tfrac{(r_{i_k,:}^k)^2}
{\|A_{i_k,:}\|^2}\ge
\tfrac12\max\limits_{i}\tfrac{(r_{i,:}^k)^2}
{\|A_{i,:}\|^2}
+\tfrac12\tfrac{\|r^k\|_F^2}{\|A\|_F^2}
$ \\[1.5mm]
&RGRK~\cite{ref20}& $
\tfrac{(r_{i_k,:}^k)^2}
{\|A_{i_k,:}\|^2}\ge
\theta\,\max\limits_{i}\tfrac{(r_{i,:}^k)^2}
{\|A_{i,:}\|^2}
+(1-\theta)\tfrac{\|r^k\|_F^2}{\|A\|_F^2}
$\\[1.5mm]
&MWRK~\cite{ref25,Mc}& $
\tfrac{(r_{i_k,:}^k)^2}
{\|A_{i_k,:}\|^2}=
\max\limits_{i}\tfrac{(r_{i,:}^k)^2}
{\|A_{i,:}\|^2}
$\\[3.5mm]
$AXB=C$
& GRBK (Alg.~\ref{ME-GRBK})& $
\tfrac{\|R_{i_k,:}^k\|^2}
{\|A_{i_k,:}\|^2}\ge
\tfrac12\max\limits_{i}\tfrac{\|R_{i,:}^k\|^2}
{\|A_{i,:}\|^2}
+\tfrac12\tfrac{\|R^k\|_F^2}{\|A\|_F^2}
$ \\[1.5mm]
& RGRBK (Alg.~\ref{ME-RGRBK})& $
\tfrac{\|R_{i_k,:}^k\|^2}
{\|A_{i_k,:}\|^2}\ge
\theta\,\max\limits_{i}\tfrac{\|R_{i,:}^k\|^2}
{\|A_{i,:}\|^2}
+(1-\theta)\tfrac{\|R^k\|_F^2}{\|A\|_F^2}
$  \\[1.5mm]
& MWRBK (Alg.~\ref{ME-MWRBK})& $
\tfrac{\|R_{i_k,:}^k\|^2}
{\|A_{i_k,:}\|^2}=
\max\limits_{i}\tfrac{\|R_{i,:}^k\|^2}
{\|A_{i,:}\|^2}
$ \\
\hline
\end{tabular}}
\end{table}

\begin{figure}[htb!]
\renewcommand{\figurename}{Fig.}
\centering
\includegraphics[scale=0.4]{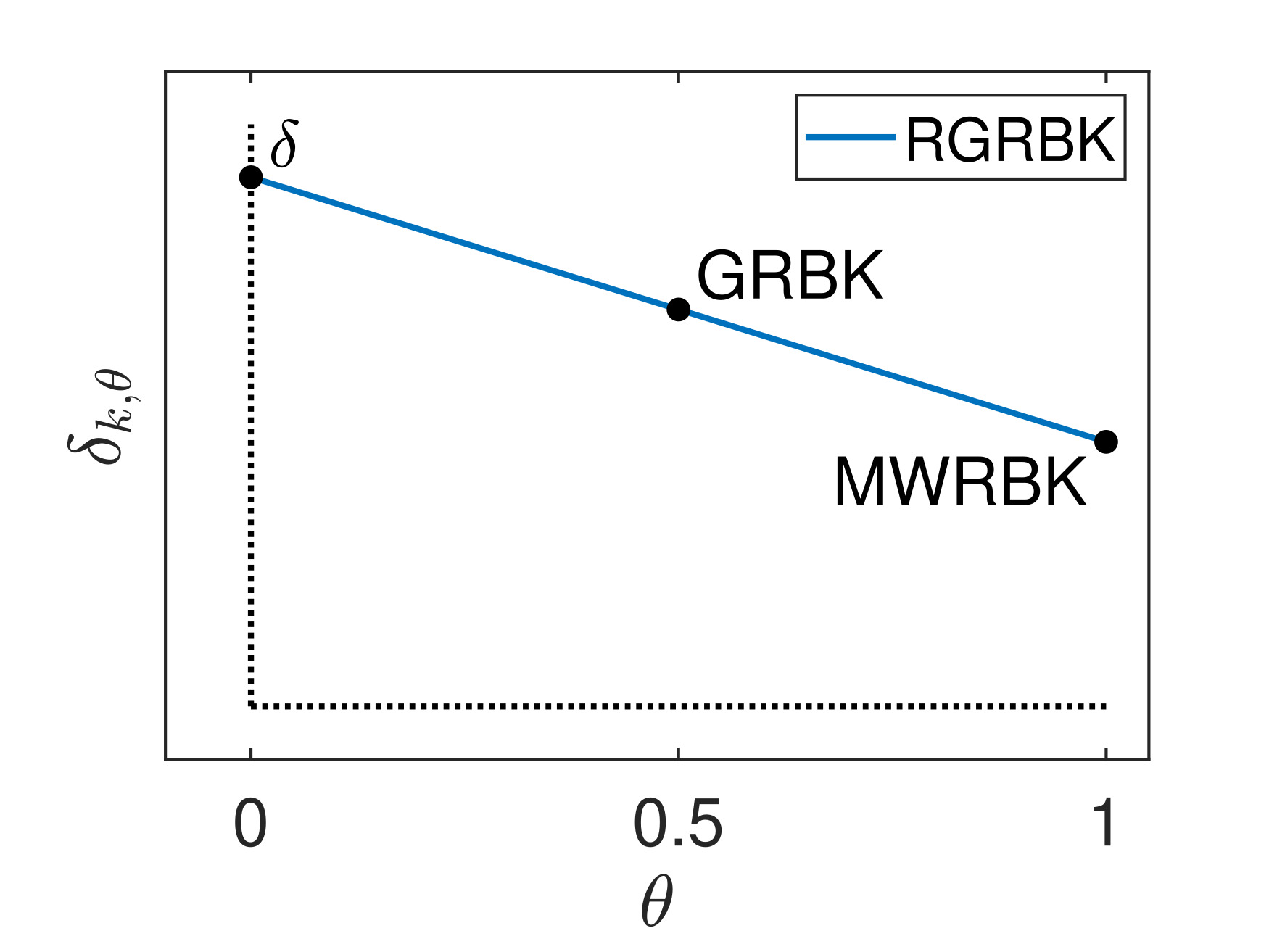}
\caption{\label{deltafig}{Schematic graph of $\delta_{k,\,\theta}$ with respect to $\theta$, and relations between RGRBK, GRBK, and MWRBK methods.}}
\end{figure}

In Table~\ref{complexity}, we summarize the per-iteration computational complexity of various methods for solving matrix equation~\eqref{eq1.1}. All algorithms, except for the GI~\cite{ref8}, are row-action methods that involve rank-one updates; cf.~also Table~\ref{problems}.

\begin{table}[htb!]
\caption{The per-iteration computational complexity of various methods for solving $AXB=C$.}
\label{complexity}
\setlength{\tabcolsep}{1.5mm}{
\begin{tabular}{lllllll}
\hline\rule{0pt}{2.3ex}%
Method & Complexity \\
\hline\rule{0pt}{2.8ex}%
BK (Alg.~\ref{ME-BK}), RBK (Alg.~\ref{ME-RBK} \cite{ref12}) & $\calo(q(p+n))$ \\ \rule{0pt}{2.8ex}%
BK (Alg.~\ref{ME-BK1}) &$\calo(q(p+n))$\\ \rule{0pt}{2.8ex}%
BK (Alg.~\ref{alg.AX=C}) &$\calo(pq)$\\ \rule{0pt}{2.8ex}%
GRBK, RGRBK, MWRBK (Algs.~\ref{ME-GRBK},~\ref{ME-RGRBK},~\ref{ME-MWRBK}) &$\calo(pq+qn+mp+mn)$\\ \rule{0pt}{2.8ex}%
RGRK~\cite{ref11}, MWRK~\cite{ref11} &$\calo(pq+qn+mp+mn)$\\ \rule{0pt}{2.8ex}%
GI~\cite{ref8} & $\min(\calo(mq(p+n)), \, \calo(pn(m+q)))$
\\[0.8ex]
\hline
    \end{tabular}}
\end{table}

\section{Numerical examples}\label{section5}

This section provides some numerical examples to validate the main theoretical findings.
All experiments are tested using MATLAB R2021b on a laptop with
Intel(R) Core(TM) i5-12500H CPU @ 2.50GHz.
We examine the methods from the aspects of the number of iterations (denoted by `IT'), elapsed CPU time in seconds (denoted by `CPU'), and relative solution error (denoted by `RSE').
In each experiment, the randomized methods (GRBK, RGRBK, and MWRBK) are executed for 20 independent trials. For randomized methods, the values of the IT and CPU are the averages of 20 trials.
We also show the standard deviation (denoted by `SD') and the range (denoted by `Range') of the CPUs derived from 20 independent experiments. Based on the discussions in Sections~\ref{section3} and \ref{section4}, we take the parameter $\alpha=\|B\|^{-2}$. In addition, as discussed in Section~\ref{section4.3}, the upper bound on the convergence factor of the RGRBK method is smaller than that of the GRBK method for $\theta\in (\tfrac12,1)$; therefore, we focus on this case.

The initial guess $X^0$ is the zero matrix for all experiments.
This is a common choice in Kaczmarz type methods, as it tends to give good approximation results.
Also, this satisfies the requirements of the various theorems (namely $\text{range}(X_0) \subseteq \text{range}(A^T)$ and $\text{range}(X_0^T) \subseteq \text{range}(B)$), so that all methods converge to the minimum norm solution $X_* = A^+ C B^+$.
The coefficient matrices $A$ and $B$ selected as test matrices include sparse matrices from the University of Florida sparse matrix collection \cite{ref33} and dense matrices randomly generated by the function $\sf randn$. These matrices have also been used in \cite{ref19, ref20, ref25, ref12, Du}.
The matrix $X$ is produced by $\sf randn(p,q)$.
To create a consistent system, the matrix $C$ is generated by $C=AXB$.
We present the size, rank, and density of the coefficient matrices, where the density is defined as
\begin{align*}
\text{density}:=
\text{number of nonzeros of an}\ m \times n\ \text{matrix} \, / \, (m\,n).
\end{align*}
We now briefly discuss possible stopping criteria.
A practical rule for all methods may be based on the norm of the rank-one update (cf.~Table~\ref{complexity}):
\[
\|X^{k+1} - X^k\| = \alpha \ \|R_{i_k,:}^k\,B^T\| \, / \, \|A_{i_k,:}\|.
\]
Another natural stopping criterion is to use $\|R^k\|$, especially in Algorithms~\ref{ME-GRBK}--\ref{ME-MWRBK}, where the residual $R^k$ is already available.
However, rather than comparing these rules in details, we focus on the behavior of the methods considering the solution error:
\begin{align*}
\text{RSE} := \|X^k-X_{\ast}\|_F\,/\,\|X_{\ast}\|_F.
\end{align*}

\begin{example}\label{example5.1} \rm
In this example, we analyze the convergence behavior of the three BK methods, i.e., Algorithms~\ref{ME-BK}, \ref{ME-BK1}, and \ref{alg.AX=C}.
We first investigate the convergence of Algorithm \ref{ME-BK}.
For a consistent matrix equation~\eqref{eq1.1}, the coefficient matrices can be categorized into nine different configurations. Specifically, when $A$ is row-full-rank, column-full-rank, or rank-deficient, $B$ may also be row-full-rank, column-full-rank, or rank-deficient.
Table~\ref{table1} lists nine corresponding sets of test matrices used to evaluate the performance of Algorithm~\ref{ME-BK}.
The convergence curves of Algorithm~\ref{ME-BK} are given in Fig.~\ref{fig1}.
Table~\ref{sizes5.1} summarizes the dimensions of the test matrices and reports the results of the IT and CPU.
These results illustrate that Algorithm~\ref{ME-BK} seems feasible for various problem sizes, including overdetermined and underdetermined systems, typically with a linear asymptotic convergence.
The time per iteration seems to have a natural relation with the sparsity of the coefficient matrices. For instance, the matrices in Sets 1 and 3 are relatively dense and have the longest per-iteration computation times.

\begin{table}[htb!]
\footnotesize
\caption{The test matrices used in Example~\ref{example5.1}, where various scenarios are tested: $m>p$ and $m<p$; $p>q$ and $p<q$; $q>n$ and $q<n$.}
\label{table1}
\setlength{\tabcolsep}{1.5mm}{
\begin{tabular}{llllllllll}
\hline\rule{0pt}{2.8ex}%
\multirow{2}{*}{Matrices}&
  & \multicolumn{2}{l}{Set 1 }& & \multicolumn{2}{l}{Set 2 }& & \multicolumn{2}{l}{Set 3}\\
\cmidrule{3-4}\cmidrule{6-7}\cmidrule{9-10}\rule{0pt}{2.3ex}%
& &$A$&$B$& &$A$ & $B$& &$A$ & $B$\\
\hline\rule{0pt}{2.8ex}%
Name& & \makecell[c]{\sf bibd$\_$11$\_$5}&{\sf bibd$\_$12$\_$4}& &{\sf bibd$\_$12$\_$4}
&{\sf ash219}&
&{\sf bibd$\_$11$\_$5} &{\sf n3c6-b2}\\
Size& &$55\times 462$&$66\times 495$& &$66\times 495$&$219\times 85$& &$55\times 462$&$455\times 105$\\
Rank& &55&66& &66&85& &55&91\\
Density& &$18.18\%$&$9.09\%$& &$9.09\%$&$2.35\%$& &$18.18\%$&$2.86\%$\\
\hline\rule{0pt}{2.8ex}%
 &
  & \multicolumn{2}{l}{Set 4}& & \multicolumn{2}{l}{Set 5}& & \multicolumn{2}{l}{Set 6}\\
\cmidrule{3-4}\cmidrule{6-7}\cmidrule{9-10}\rule{0pt}{1.3ex}%
Name& &{\sf ash219}&{\sf bibd$\_$12$\_$4}&
&{\sf ash219}&{\sf ash331}&
&{\sf ash219}&{\sf n3c6-b1}\\
Size& &$219\times 85$&$66\times 495$&
&$219\times 85$&$331\times 104$&
&$219\times 85$&$105\times 105$\\
Rank& &85&66& &85&104& &85&14\\
Density& &$2.35\%$&$9.09\%$& &$2.35\%$&$1.92\%$& &$2.35\%$&$1.90\%$\\
\hline\rule{0pt}{2.8ex}%
 &
  & \multicolumn{2}{l}{Set 7}& & \multicolumn{2}{l}{Set 8}& & \multicolumn{2}{l}{Set 9}\\
\cmidrule{3-4}\cmidrule{6-7}\cmidrule{9-10}\rule{0pt}{1.0ex}%
Name& &{\sf n3c6-b1}&{\sf lp$\_$grow7}&
&{\sf n3c6-b2}&{\sf ash219}&
&{\sf n3c6-b2}&{\sf cis-n4c6-b1}\\
Size& &$105\times 105$&$140\times 301$&
&$455\times 105$&$219\times 85$&
&$455\times 105$&$210\times 21$\\
Rank& &14&140& &91&85& &91&20\\
Density& &$1.90\%$&$6.20\%$& &$2.86\%$&$2.35\%$& &$2.86\%$&$9.52\%$\\
\hline
    \end{tabular}}
\end{table}

\begin{figure}[htb!]
\renewcommand{\figurename}{Fig.}
\centering
\includegraphics[scale=0.5]{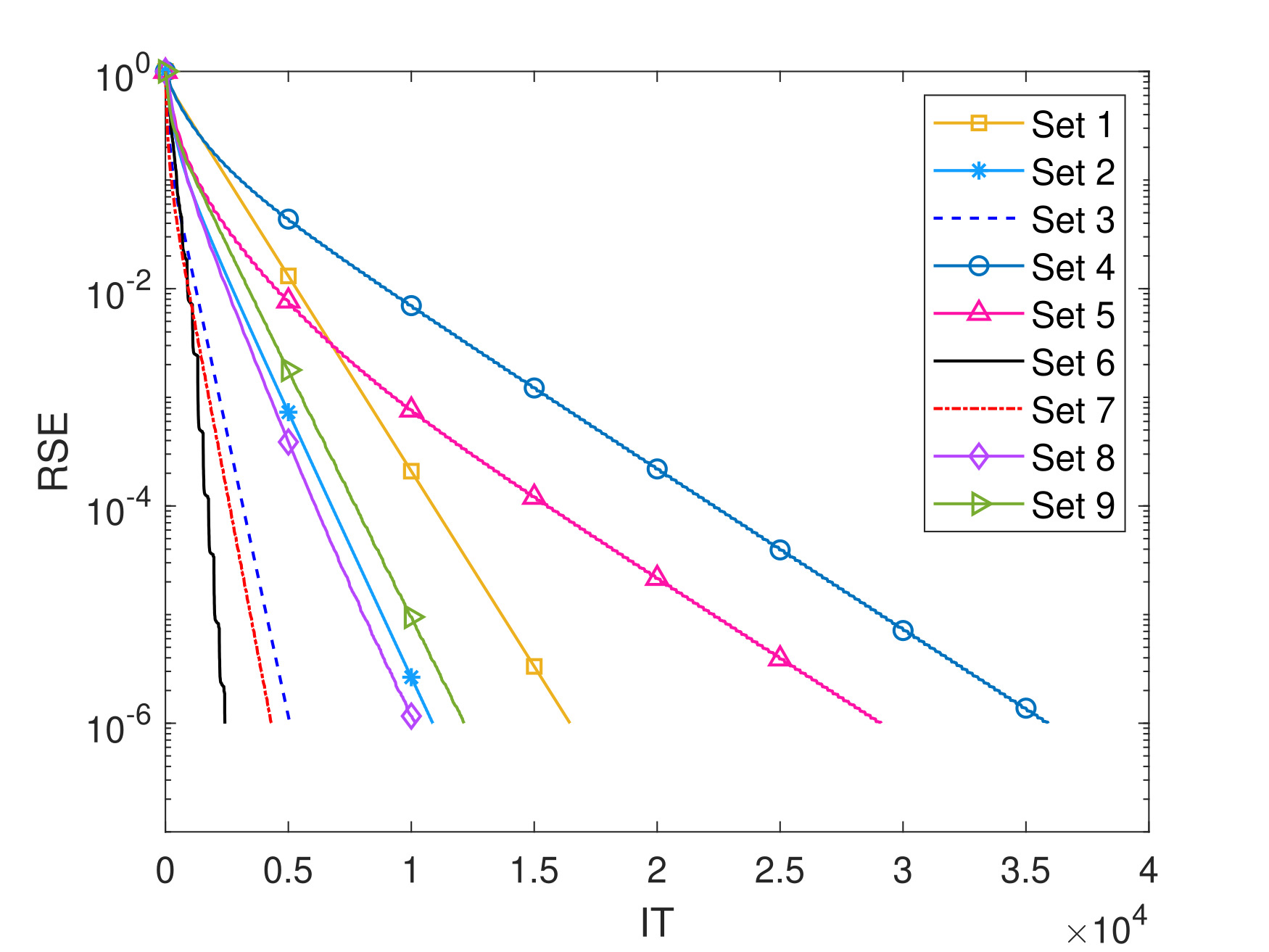}
\caption{Convergence curves of Algorithm~\ref{ME-BK} showing RSE versus IT for the test matrices listed in Table~\ref{table1}.}
\label{fig1}
\end{figure}

\begin{table}[htb!]
\footnotesize
\caption{The results for IT, CPU, and CPU/IT corresponding to Fig.~\ref{fig1}.}
\label{sizes5.1}
\setlength{\tabcolsep}{1.5mm}{
    \begin{tabular}{llllllllll}
\hline\rule{0pt}{2.8ex}%
Set &1&2&3&4&5\\ \hline\rule{0pt}{2.8ex}%
IT&16444&10855&5029&35893&27139\\
CPU&15.98&2.46&5.30&4.41&2.15\\
CPU/IT &$9.7\cdot10^{-4}$&$2.3\cdot10^{-4}$&$1.1\cdot10^{-3}$&$1.2\cdot10^{-4}$&$7.9\cdot10^{-5}$\\ \hline\rule{0pt}{2.8ex}%
Set &6&7&8&9\\ \cmidrule{1-5}
IT&2427&4162&10363&12384\\
CPU&0.08&0.53&0.60&0.66\\
CPU/IT &$3.4\cdot10^{-5}$&$1.3\cdot10^{-4}$&$5.8\cdot10^{-5}$&$5.3\cdot10^{-5}$\\
\cmidrule{1-5}
\end{tabular}}
\end{table}

We next investigate the convergence behavior of Algorithms~\ref{ME-BK1} and \ref{alg.AX=C}. We select Set 5 (full column rank of $B$ for Algorithm~\ref{ME-BK1}) and Set 1 (full row rank of $B$ for Algorithm~\ref{alg.AX=C}) from Table~\ref{table1}.
Fig.~\ref{BK_fullcolumnrow} displays the spectral radius as function of $\alpha$, the convergence plots (RSE versus IT) for various $\alpha$, and the plots of $\rho^k$ and RSE versus IT.
Table~\ref{algs.34} presents the IT and CPU corresponding to the middle row of Fig.~\ref{BK_fullcolumnrow}.
Note that in Fig.~\ref{BK_fullcolumnrow} and Table~\ref{algs.34}, ``IT" denotes the number of sweeps, whereas in all other tables and figures, ``IT" refers to the number of iteration steps.
These results confirm that the methods perform better with smaller $\rho$.
From their spectral radius curves, we observe that $\rho<1$ when $\alpha\in(0,2)$, which is consistent with Theorems~\ref{theorem3.7} and \ref{BK_full_row}.

\begin{figure}[htb!]
\renewcommand{\figurename}{Fig.}
\centering
\includegraphics[scale=1]{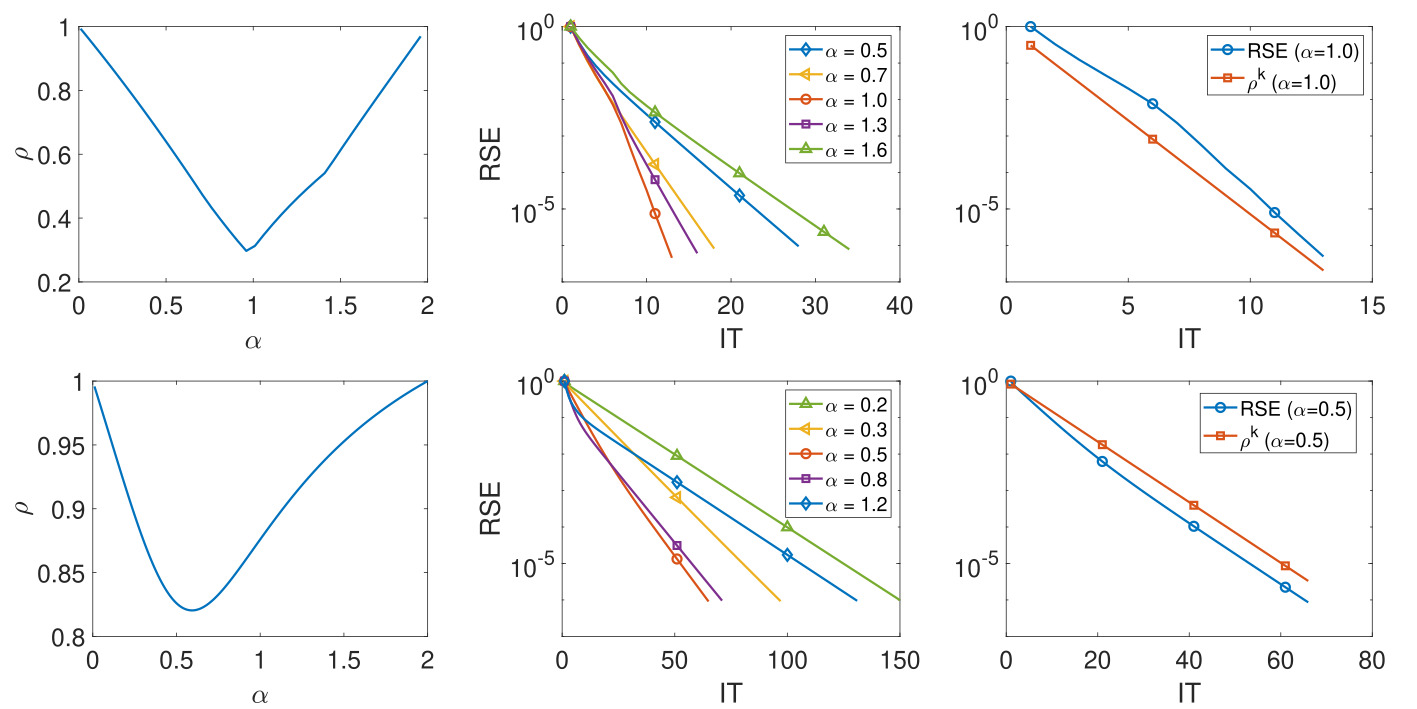}
\caption{For Set~5 (top row) and Set~1 (bottom row): the spectral radius, the convergence (RSE versus IT) for different $\alpha$, and the plots of $\rho^k$ and RSE versus IT for Algorithms \ref{ME-BK1} (top row) and \ref{alg.AX=C} (bottom row).
Note that the {\em upper bound} on $\rho$ is minimized for $\alpha=1$ in these cases, but the actual minimizer $\alpha$ might differ from this value. Only in this figure, IT means the number of sweeps (rather than row-action steps).
}
\label{BK_fullcolumnrow}
\end{figure}

\begin{table}[h]
\footnotesize
\caption{The results for IT and CPU in the middle column of Fig.~\ref{BK_fullcolumnrow}.}
\label{algs.34}
\setlength{\tabcolsep}{1.5mm}{
\begin{tabular}{ccccccc}
\hline \noalign{\vskip 3pt}
\multicolumn{3}{c}{Alg.~\ref{ME-BK1}}& &\multicolumn{3}{c}{Alg.~\ref{alg.AX=C}}\\
\cmidrule{1-3}\cmidrule{5-7}\rule{0pt}{1.8ex}%
$\alpha$&IT&CPU&&$\alpha$&IT&CPU\\\hline \rule{0pt}{2.5ex}%
0.5&27&0.38&&0.2&150&0.50\\
0.7&17&0.23&&0.3&\ph196&0.30\\
1.0&12&0.16&&0.5&\ph165&0.21\\
1.3&15&0.20&&0.8&\ph171&0.23\\
1.6&33&0.43&&1.2&131&0.41\\
\hline
\end{tabular}}
\end{table}
\end{example}

\begin{example}\label{example5.3} \rm
In this example, we compare the convergence behavior of the RBK~\cite{ref12}, GRBK, RGRBK, and MWRBK methods.
Three configurations of the coefficient matrices $A$ and $B$ are tested: (1) full column rank $A$ and full row rank $B$ (Sets 1 and 2 in Table~\ref{table5.31}), (2) full row rank $A$ and full column rank $B$ (Sets 3 and 4 in Table~\ref{table5.31}), and (3) rank-deficient $A$ and $B$ (Sets 5 and 6 in Table~\ref{table5.31}).

\begin{table}[h]
\footnotesize
\caption{The coefficient matrices used in Example~\ref{example5.3}.}
\label{table5.31}
\setlength{\tabcolsep}{1.5mm}{
\begin{tabular}{llllllllll}
\hline\rule{0pt}{2.8ex}%
\multirow{2}{*}{Matrices}&
  & \multicolumn{2}{l}{Set 1}& & \multicolumn{2}{l}{Set 2}\\
\cmidrule{3-4}\cmidrule{6-7}\rule{0pt}{2.3ex}%
& &$A$&$B$& &$A$ & $B$\\
\hline\rule{0pt}{2.8ex}%
Name& &{\sf ash608}& {\sf bibd$\_$15$\_$3} & &{\sf \makecell[l]{randn(455, 20)}}&{\sf \makecell[l]{randn(80, 320)}}\\
Size& &$608\times 188$&$105\times 455$& &$455\times 20$&$80\times 320$\\
Rank& &188&105& &20&80\\
Density& &$1.06\%$&$2.86\%$& &$100\%$&$100\%$\\
\hline\rule{0pt}{2.8ex}%
& & \multicolumn{2}{l}{Set 3}&& \multicolumn{2}{l}{Set 4}\\
\cmidrule{3-4}\cmidrule{6-7}\rule{0pt}{2ex}%
Name&&{\sf lp$\_$afiro}&{\sf ash219}& &{\sf \makecell[l]{randn(35, 60)}}&{\sf \makecell[l]{randn(80, 20)}}\\\rule{0pt}{2.3ex}%
Size& &$27\times 51$&$219\times 85$& &$35\times 60$&$80\times 20$\\
Rank& &27&85& &35&20\\
Density& &$7.41\%$&$2.35\%$& &$100\%$&$100\%$\\
\hline\rule{0pt}{2.8ex}%
& &  \multicolumn{2}{l}{Set 5}& &  \multicolumn{2}{l}{Set 6}\\
\cmidrule{3-4}\cmidrule{6-7}\rule{0pt}{2.5ex}%
Name& &{\sf flower$\_$4$\_$1}&{\sf n3c6-b2}& & {\sf \makecell[l]{A=randn(265, 25);\\ A=[A,A]}}&{\sf \makecell[l]{B=randn(10, 345);\\B=[B;B]}}\\\rule{0pt}{2.3ex}%
Size& &$121\times 129$&$455\times 105$& & $265\times 50$&$20\times 345$\\
Rank& &108&91& &25&10\\
Density& &$2.47\%$&$2.86\%$& &$100\%$&$100\%$\\
\hline
\end{tabular}}
\end{table}

Table~\ref{table4} provides the comparative results of the IT, CPU, SD, Range, and CPU/IT, using the test coefficient matrices in Table~\ref{table5.31}.
As shown in Table~\ref{table4}, although the RBK method has the shortest average time per iteration, GRBK, RGRBK, and MWRBK require fewer iterations and less total computation time.
In addition, the randomized algorithms have small SD, with GRBK and RGRBK generally showing slightly lower SD values than RBK;
this is natural since they select from fewer rows.

\begin{table}[htb!]
\footnotesize
\caption{Numerical results of the RBK, GRBK, RGRBK, and MWRBK methods for the test coefficient matrices in Table~\ref{table5.31} (with $m$, $p$, $q$, $n$ values).}
\label{table4}
\setlength{\tabcolsep}{0.99mm}{
\begin{tabular}{lcccccccccccc}
\hline\rule{0pt}{2.8ex}%
\multirow{2}{*}{Method}&
  & \multicolumn{5}{c}{Set 1 \ (608, 188, 105, 455) }& & \multicolumn{5}{c}{Set 2 \ (455, 20, 80, 320)}\\
\cmidrule{3-7}\cmidrule{9-13}\rule{0pt}{1.8ex}%
 & &IT&CPU&SD&Range&CPU/IT& &IT&CPU&SD&Range&CPU/IT\\
\hline\rule{0pt}{2.8ex}%
RBK& & 33891&29.9&2.17&[25.5, 33.3]&$8.8\cdot10^{-4}$& & 2589&1.50&0.05&$[1.38, 1.57]$&$5.8\cdot10^{-4}$\\
GRBK & & 12736&21.9&0.04&[21.9, 22.1]&$1.7\cdot10^{-3}$& & 1417&1.33&0.01&$[1.30, 1.35]$&$9.4\cdot10^{-4}$\\
RGRBK& & 12715&21.7&0.04&[21.6, 21.7]&$1.7\cdot10^{-3}$& & 1399&1.28&0.02&[1.27, 1.34]&$9.2\cdot10^{-4}$\\
MWRBK& & 12711&17.1&--&--&$1.3\cdot10^{-3}$& & 1398&1.18&--&--&$8.4\cdot10^{-4}$\\
\hline\rule{0pt}{2.8ex}%
& &  \multicolumn{5}{c}{Set 3 \ (27, 51, 219, 850)}& &  \multicolumn{5}{c}{Set 4 \ (35, 60, 80, 20)}\\
\cmidrule{3-7}\cmidrule{9-13}\rule{0pt}{1.8ex}%
RBK& & 30897&1.87&0.34&[1.62, 2.65]&$6.1\cdot10^{-5}$& & 32702&0.46&0.05&[0.40, 0.60]&$1.4\cdot10^{-5}$\\
GRBK& & 13229&0.89&0.02&[0.86, 0.96]&$6.7\cdot10^{-5}$& & 23519&0.46&0.02&[0.43, 0.49]&$2.0\cdot10^{-5}$\\
RGRBK& & 13219&0.88&0.02&[0.85, 0.93]&$6.6\cdot10^{-5}$& & 23512&0.46&0.03&[0.42, 0.53]&$2.0\cdot10^{-5}$\\
MWRBK& & 13213&0.79&--&--&$6.0\cdot10^{-5}$& & 23512&0.37&--&--&$1.6\cdot10^{-5}$\\
\hline\rule{0pt}{2.8ex}%
& &  \multicolumn{5}{c}{Set 5 \ (121, 129, 455, 105)}& &  \multicolumn{5}{c}{Set 6 \ (265, 50, 20, 345)}\\
\cmidrule{3-7}\cmidrule{9-13}\rule{0pt}{1.8ex}%
RBK& & 32859&8.03&0.14&[7.87, 8.42]&$2.4\cdot10^{-4}$& & 1043&0.16&0.01&[0.14, 0.18]&$1.5\cdot10^{-4}$\\
GRBK& & 11071&3.64&0.04&[3.58, 3.76]&$3.3\cdot10^{-4}$& &  \phantom{1}388&0.14&0.01&[0.12, 0.16]&$3.6\cdot10^{-4}$\\
RGRBK& & 11029&3.42&0.03&[3.37, 3.48]&$3.1\cdot10^{-4}$& &  \phantom{1}363&0.13&0.01&[0.13, 0.15]&$3.7\cdot10^{-4}$\\
MWRBK& & 11012&2.85&--&--&$2.6\cdot10^{-4}$& &  \phantom{1}355&0.11&--&--&$3.1\cdot10^{-4}$\\
\hline
\end{tabular}}
\end{table}
\end{example}

\section{An application to color image restoration}\label{section6}

Many imaging systems, such as magnetic resonance imaging, computerized tomography, etc., can be modeled by the system of linear equations~\eqref{eq1.4}~\cite{ref16,ref17}.
Similarly, matrix equation~\eqref{eq1.1} plays a significant role in certain color image restoration problems~\cite{ref5}.

We first give a description of the forward model for color image restoration; see also~\cite{ref5}.
The desired original color image $\calx$ (the observed color image $\mathcal{C}$) can be shown as a three-dimensional array of size $m\times n\times 3$, where the red (R), green (G) and blue (B) channels are correspondingly three matrices $X_{{\text r}}:=\calx(:,:,1)$, $X_{{\text g}}:=\calx(:,:,2)$, $X_{{\text b}}:=\calx(:,:,3)$
(and similarly, $C_{{\text r}}:=\mathcal{C}(:,:,1)$, $
C_{{\text g}}:=\mathcal{C}(:,:,2)$, $
C_{{\text b}}:=\mathcal{C}(:,:,3)$).
With $X=(\text{vec}(X_{{\text r}}),\text{vec}(X_{{\text g}}),\text{vec}(X_{{\text b}}))$ and $
C=(\text{vec}(
C_{{\text r}}),\text{vec}(
C_{{\text g}}),\text{vec}(
C_{{\text b}}))$, the forward model is expressed as
$C=AXA_{\rm c}^T+E$, where $A\in \R^{mn\times mn}$ is the within-channel blurring, $A_{\rm c}\in \R^{3\times 3}$ represents the cross-channel blurring and $E\in \R^{mn\times 3}$ is the additive noise. Ideally, if there is no noise in the generation of the blurred image, the forward model becomes~(\ref{eq1.3}).
In this paper, we only consider this noise-free case.

We take {\sf face} of size $92\times92\times3$, {\sf bird} of size $96\times96\times3$ and {\sf mandril} of size $125\times120\times3$
as test images, as shown in the top row of Fig.~\ref{test and blurred}, to investigate the performance of the following iterative methods: GI~\cite{ref8}, RBK~\cite{ref12}, and our proposed BK, GRBK, RGRBK, and MWRBK methods.
To obtain blurred images, the Matlab function {\sf fspecial} is used to generate a rotationally symmetric Gaussian lowpass filter of size $5$ with standard deviation $6$
to simulate the blurred images caused by the Gaussian filter.
The cross-channel blurring matrix $A_{\rm c}$ captures the relations among different color channels (i.e., RGB) that may arise due to the blurring. In typical systems, the dominant contribution to a channel originates from itself, implying $(A_{\rm c})_{i,i}>(A_{\rm c})_{i,j}$ for $i\ne j$. To ensure the preservation of total intensity, each row of $A_{\rm c}$ is constrained to sum to 1, i.e., $\sum_{j}(A_{\rm c})_{i,j}=1$.
In our examples, we choose the (rather typical) cross-channel blurring matrix as
\begin{align} \small \label{A_c}
A_{\rm c}=\left[\begin{array}{cccc}
0.90 &0.05 &0.05\\
0.00 &0.90 &0.10\\
 0.05 &0.10 &0.85
        \end {array}
\right].
\end{align}
The bottom row of Fig.~\ref{test and blurred} shows the blurred images, with their corresponding peak signal-to-noise ratio (PSNR) values 19.31, 17.52, and 18.86, respectively.

\begin{figure}[htb!]
\renewcommand{\figurename}{Fig.}
\center{\includegraphics[scale=1.6]{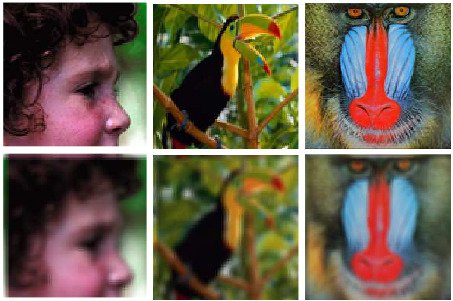}}
\caption{\label{test and blurred} Test images (top row) and blurred images (bottom row).}
\end{figure}

Table~\ref{sizes} provides an overview of the size of the experiments in this section.
Notice that for this application, $m=p$ (i.e., $A$ is square), and $q=n$ ($B$ is square), with $p \gg q$ ($X$ is tall and skinny).
We remark that in this specific experiment the matrix $B=A_c^T$ is a small $3\times 3$ nonsingular matrix, for which direct inversion is natural and offers a natural solution, as in Algorithm~\ref{alg.AX=C}.
We test this approach in Experiment~\ref{exp6.2}, but also stress that we have developed methods to solve (\ref{eq1.1}) for $A$ and $B$ of arbitrary size.

\begin{table}[htb!]
\footnotesize
\caption{Overview of the sizes of the experiments in Section~\ref{section6}, and of the degrees of freedom (dofs) in the solution $X$ (equal to $pq$).}
\label{sizes}
\setlength{\tabcolsep}{3mm}{
\begin{tabular}{lccc}
\hline\rule{0pt}{2.3ex}%
Experiment & $m = p$ & $q = n$ & dofs \\ \hline\rule{0pt}{2.8ex}%
{\sf face}    & \ph18464 & 3 & 25392 \\
{\sf bird}    & \ph19216 & 3 & 27648 \\
{\sf mandril} &    15000 & 3 & 45000 \\
\hline
    \end{tabular}}
\end{table}

\begin{experiment}\label{exp6.1} \rm
In this experiment, we compare the restored results of the GI~\cite{ref8}, BK (Alg.~\ref{ME-BK}), RBK~\cite{ref12} (Alg.~\ref{ME-RBK}), GRBK (Alg.~\ref{ME-GRBK}), RGRBK (Alg.~\ref{ME-RGRBK}), and MWRBK (Alg.~\ref{ME-MWRBK}) methods. All methods, except for GI, process one row per iteration. The GI method uses all rows in each iteration, which may be computationally more expensive, as also confirmed by this experiment.

In Fig.~\ref{PSNR_CPU}, we plot the PSNR versus the time spent using these methods. In this experiment, we run the randomized methods only once.
To achieve comparable PSNR values, the GI, BK, and RBK methods require more computation time than the GRBK, RGRBK, and MWRBK methods, with the GI method being the most time-consuming among all.
As expected from the discussion in Section~\ref{section4}, increasing $\theta$ from 0.5 (GRBK), via an intermediate value in the interval $(0.5, 1)$ (RGRBK), to 1 (MWRBK) has a favorable effect on the performance.
Furthermore, the randomized RBK is still faster than the deterministic BK, which means that taking rows with larger norm pays off.
When we would scale \eqref{eq1.1} apriori such that all rows have equal norm, it is expected that BK and RBK behave similarly.

\begin{figure}[h]
\renewcommand{\figurename}{Fig.}
\centering
\subfigure
{  \hspace{-6mm}
 	\begin{minipage}[b]{.32\linewidth}
        \centering
        \includegraphics[scale=0.039]{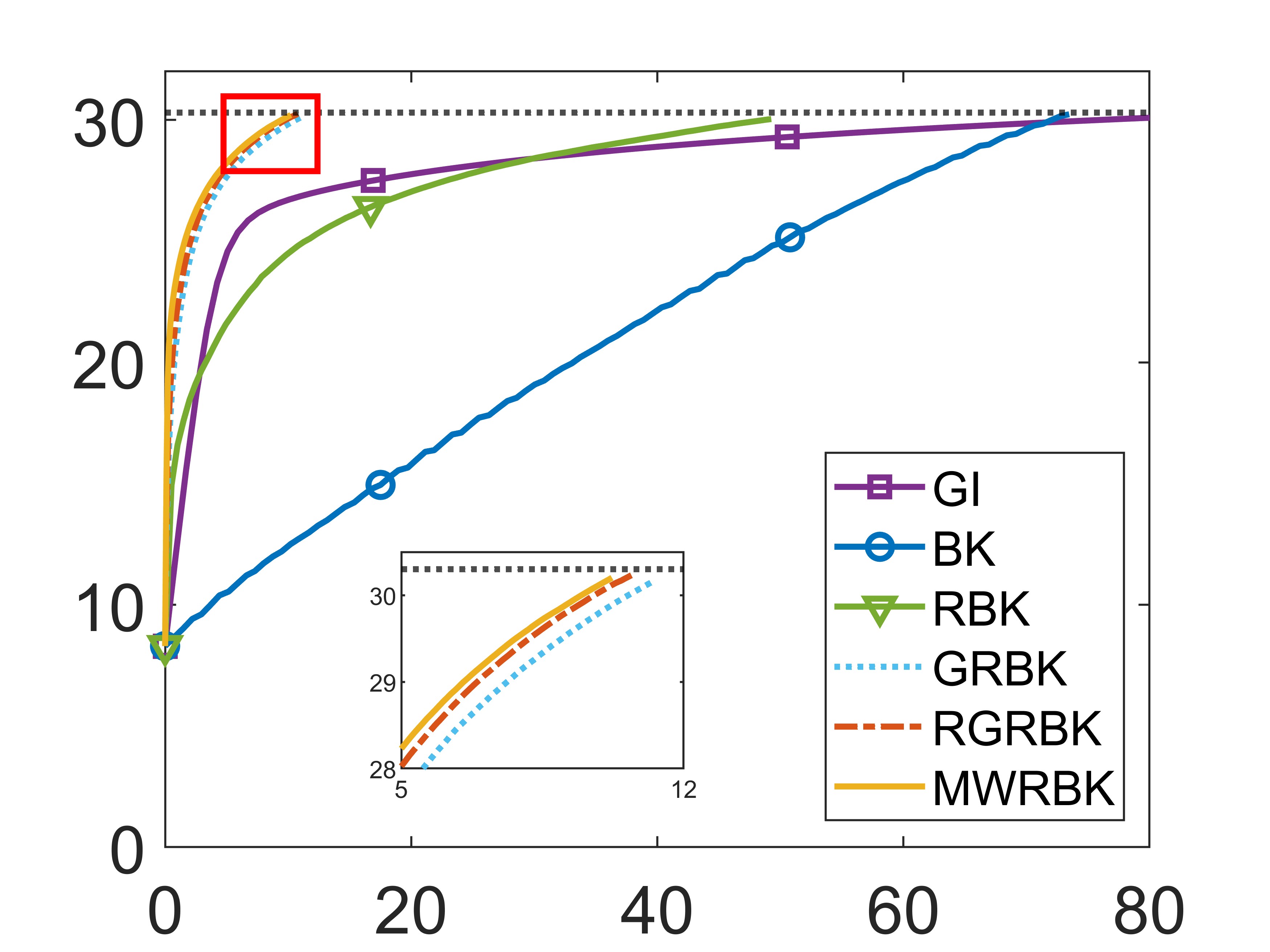}
    \end{minipage}
}
\subfigure
{
 	\begin{minipage}[b]{.31\linewidth}
        \centering
        \includegraphics[scale=0.039]{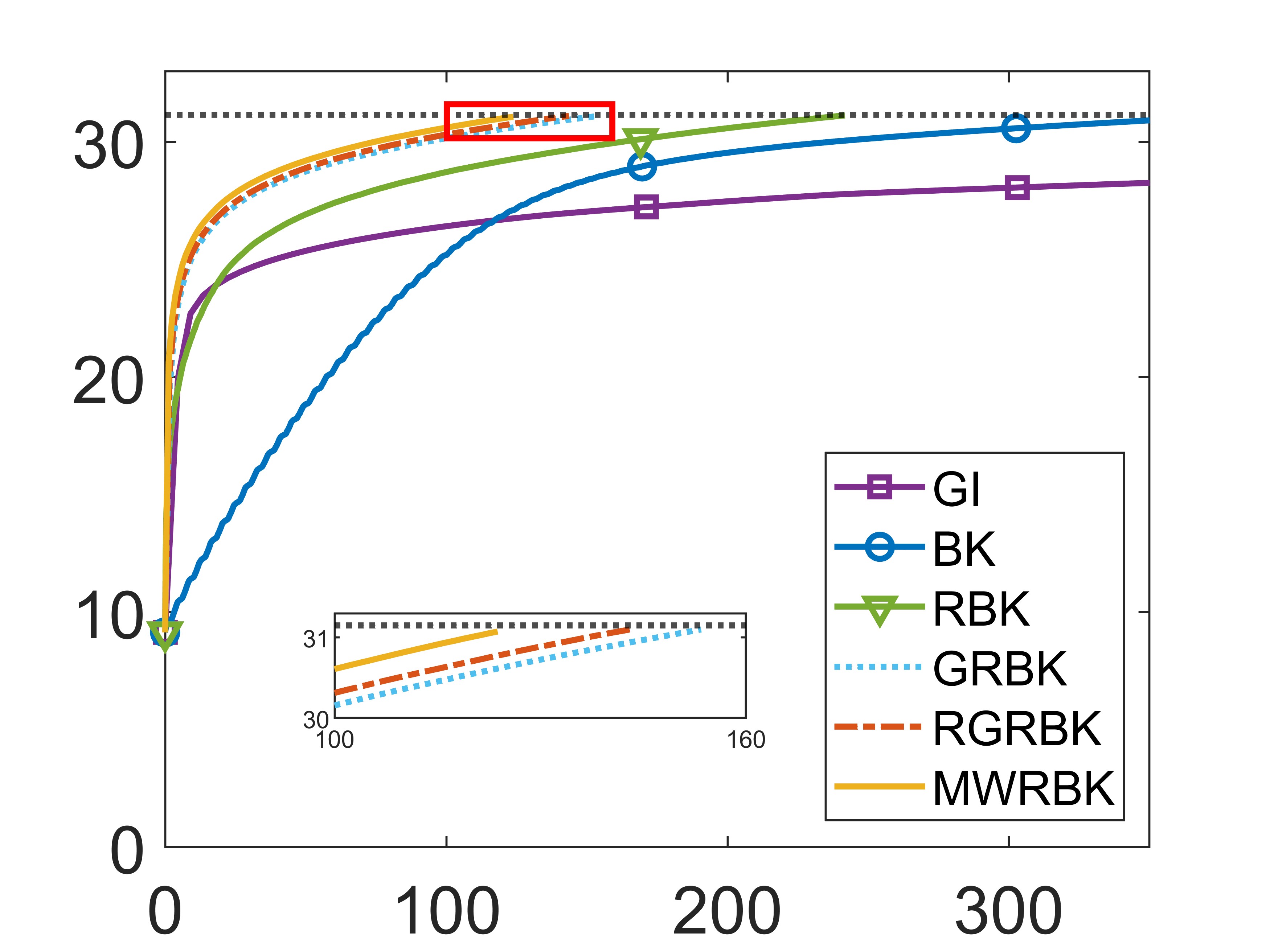}
    \end{minipage}
}
\subfigure
{
 	\begin{minipage}[b]{.31\linewidth}
        \centering
        \includegraphics[scale=0.039]{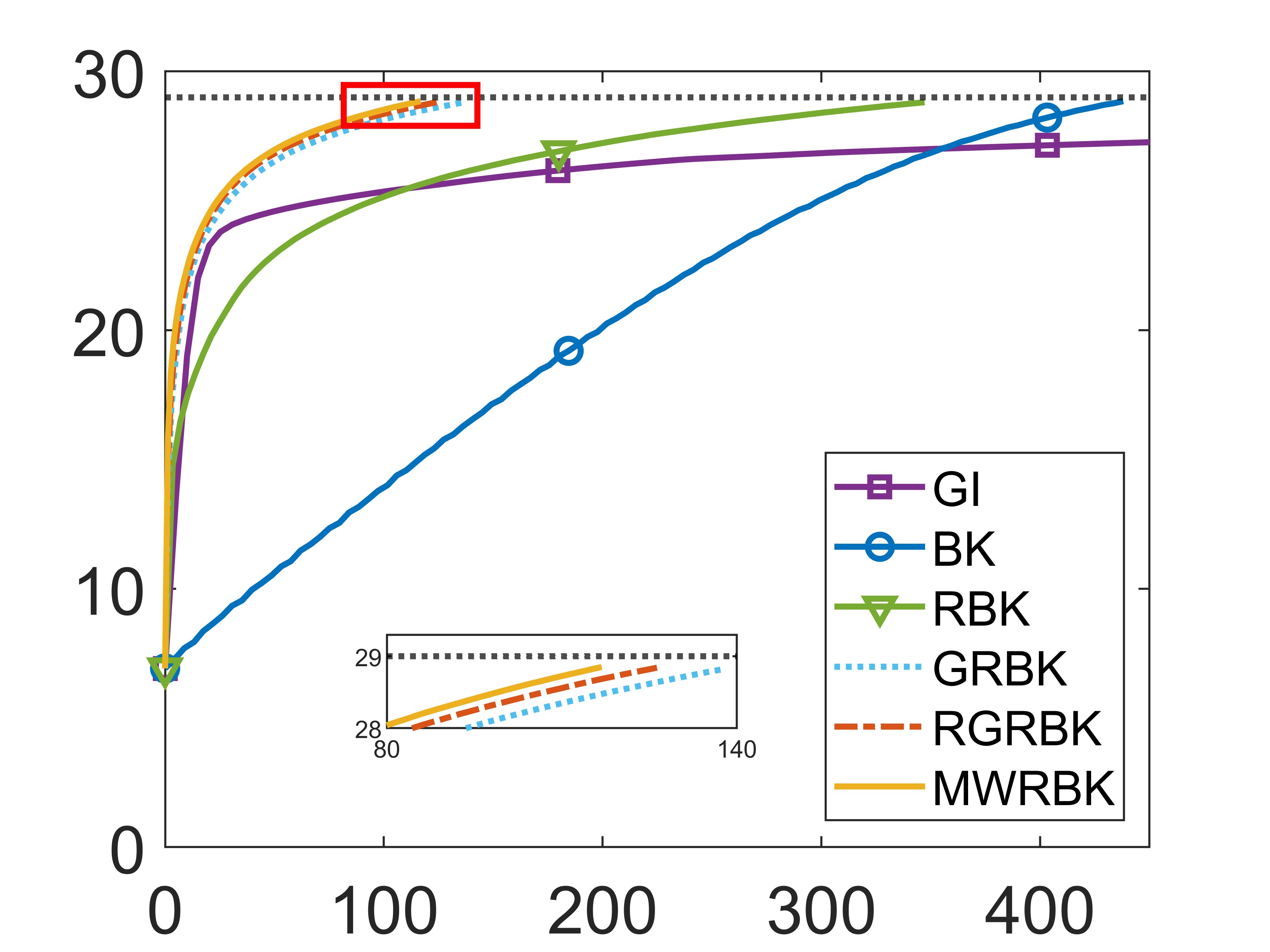}
    \end{minipage}
}
\caption{\label{PSNR_CPU}{Relation between PSNR and CPU for the restored {\sf face} (left), {\sf bird} (middle), and {\sf mandril} (right) images using different methods.}}
\end{figure}

\end{experiment}

\begin{experiment}\label{exp6.2} \rm
In this experiment, we compare the recovery results obtained by the GI~\cite{ref8}, BK (Alg.~\ref{ME-BK}), RBK~\cite{ref12}, GRBK, RGRBK, and MWRBK methods. For a fair comparison and a good impression of the typical behavior of the various methods, we use the exact solution in the stopping criterion and terminate the iterations once $\text{RSE}\le 8\cdot10^{-2}$.
The PSNR values of the restored images obtained by different algorithms are identical: 30.23 ({\sf face}), 31.08 ({\sf bird}), and 28.85 ({\sf mandril}), respectively.
Table \ref{table10} presents the numerical results, and Fig.~\ref{restored} shows the restored images obtained by the MWRBK method.
We observe that the GI method requires the longest per-iteration and total computation time, which is consistent with its cubic computational complexity compared to the quadratic complexity of the other algorithms. Although the number of iterations required by GI varies across cases, it seems a less competitive method overall.

Although GRBK, RGRBK, and MWRBK incur higher CPU time per iteration (CPU/IT) than BK and RBK, they require significantly fewer iterations, resulting in a shorter overall computation time than BK and RBK. These results show that their overall convergence is faster than that of BK and RBK.

As the problem size increases, the SD of the randomized algorithms (RBK, GRBK, and RGRBK) increases slightly. For example, in experiment $\sf face$, the SD values of the randomized methods are less than $2$, while in experiment $\sf mandrill$, the SD values increase modestly but remain below $7$.

In addition, in unreported experiments we have compared the behavior of three BK methods (Algorithms~\ref{ME-BK}, \ref{ME-BK1}, and \ref{alg.AX=C}).
For the case of this $3 \times 3$ matrix $A_{\text{c}}$, these BK methods generally have very similar number of iterations and runtimes, irrespective of how close $A_{\text{c}}$ is to the identity.
Proposition~\ref{Alg._3_and_4_are_equivalent} shows that Algorithms \ref{ME-BK1} and \ref{alg.AX=C} are equivalent given the fact that $A_{\text{c}}$ is nonsingular, and our results indeed confirm that the numbers of iterations are identical.

\begin{table}[htb!]
\footnotesize
\caption{The PSNR, IT, CPU, SD, Range, and CPU/IT of the restored images by the various methods, where all methods, except for GI, use a rank-one update at each iteration.}
\label{table10}
\setlength{\tabcolsep}{1.5mm}{
\begin{tabular}{lccccccc}
\hline\rule{0pt}{2.5ex}%
\multirow{2}{*}{Method}&
  & \multicolumn{6}{c}{{\sf face}}\\
\cmidrule{3-8}\rule{0pt}{2.0ex}%
& &PSNR&IT&CPU&SD&Range&CPU/IT\\
 \hline\rule{0pt}{2.8ex}%
GI & &  30.23& \ph137689& 113.8& --& --& $3.0\cdot 10^{-3}$\\
BK & & 30.23&281401&105.2&--&--&$3.7\cdot 10^{-4}$\\
RBK & & 30.23&139442& \ph167.0&1.73&[64.1, 71.3]&$4.8\cdot 10^{-4}$\\
GRBK & & 30.23& \phantom{1}11534& \ph119.0&0.58&[18.1, 20.0]&$1.7\cdot 10^{-3}$\\
RGRBK & & 30.23& \phantom{1}11191& \ph117.8&0.70&[16.7, 18.9]&$1.6\cdot 10^{-3}$\\
MWRBK & & 30.23& \phantom{1}12045& \ph117.4&--&--&$1.5\cdot 10^{-3}$\\
\hline\rule{0pt}{2.5ex}%
&
  & \multicolumn{6}{c}{{\sf bird}}\\
\cmidrule{3-8}\rule{0pt}{1.5ex}%
GI & &  31.08& 602495& 1919.3& --&--& $3.2\cdot 10^{-3}$\\
BK & & 31.08&821648& \ph1330.0&--&--&$4.0\cdot 10^{-4}$\\
RBK & & 31.08&390479& \ph1207.8&4.11&[200.5, 217.2]&$5.3\cdot 10^{-4}$\\
GRBK & & 31.08& \phantom{1}97871& \ph1167.8&3.09&[163.1, 173.2]&$1.7\cdot 10^{-3}$\\
RGRBK & & 31.08& \phantom{1}96407& \ph1158.4&2.10&[154.6, 162.1]&$1.6\cdot 10^{-3}$\\
MWRBK & & 31.08& \phantom{1}97475& \phantom{1}147.2&--&--&$1.5\cdot 10^{-3}$\\
\hline\rule{0pt}{2.5ex}%
&
  & \multicolumn{6}{c}{{\sf mandril}}\\
\cmidrule{3-8}\rule{0pt}{1.5ex}
GI & &  28.85& 254358& 1224.6& --& --& $4.8\cdot 10^{-3}$\\
BK & & 28.85&595118& \ph1474.2& --&--&$8.0\cdot 10^{-4}$\\
RBK & & 28.85&375388& \ph1314.1& 6.41&[299.1, 328.7]&$8.4\cdot 10^{-4}$\\
GRBK & & 28.85& \phantom{1}59750& \ph1153.8& 2.29&[150.6, 158.3]&$2.6\cdot 10^{-3}$\\
RGRBK & & 28.85& \phantom{1}58563& \ph1145.1& 2.52&[141.0, 149.4]&$2.5\cdot 10^{-3}$\\
MWRBK & & 28.85& \phantom{1}58522& \phantom{1}123.8& --&--&$2.1\cdot 10^{-3}$\\
\hline
    \end{tabular}}
\end{table}

\begin{figure}[htb!]
\renewcommand{\figurename}{Fig.}
\center{\includegraphics[scale=1.6]{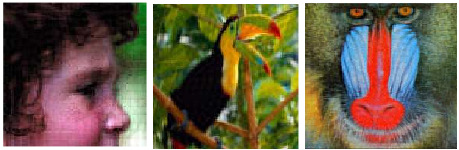}}
\caption{Restored images obtained by the MWRBK method for Experiment \ref{exp6.2}.}\label{restored}
\end{figure}
\end{experiment}

Based on the results provided in Tables~\ref{sizes} and \ref{table10}, we observe a proportional relation between the dofs and the computational cost per iteration (CPU/IT) for the three test images.
For instance, in the RBK method, the CPU/IT increases from $8.4 \cdot 10^{-4}$ for {\sf face} to $4.8 \cdot 10^{-4}$ for {\textsf mandril}, a ratio of approximately $1.75$. The dofs of {\textsf mandril} are approximately $1.77$ times those of {\sf face}.
This consistent trend shows the influence of the problem size on per-iteration computational cost.
The example {\sf bird} requires the largest number of iterations, which might be due to the many edges in the figure.
There is no clear relation of the number of iterations with the condition number of the within-channel blurring matrices $A$, as {\sf bird} corresponds to the matrix having the smallest condition number of the three examples.

Finally, some words about the scaling of the methods for larger images.
We take larger $140\times 140$ and $280\times 280$ versions of {\sf face};
with $140^2 \cdot 3 = 58800$ and $280^2 \cdot 3 = 235200$ dofs, these have ca 2.3 and 9.3 more unknowns than the original $92 \times 92$ image.
The CPU/IT scales exactly linear with the dofs.
The number of iterations increases for larger images, probably because of having more details.
For the $140\times 140$ {\sf face}, IT is ca 19000 instead of 12000.
In conclusion, the CPU/IT is proportional to the dofs, and the number of iterations seems to depend on the complexity of the image.

\section{Conclusions}\label{section7}
We have explored deterministic and randomized Kaczmarz methods to solve the consistent matrix equation \eqref{eq1.1}, which arises commonly in many applications.
Our methods can be used for $A$ and $B$ of any size (i.e., overdetermined and underdetermined cases), with the exception of Algorithms~\ref{ME-BK1} and~\ref{alg.AX=C}, which require $B$ to have full column or row rank.
The gradient type GI method in \cite{ref8} uses a full-rank update in each iteration.
The methods in \cite{ref11} use a rank-one addition with a single element of $C$ and one column of $B$ in each step.
All of our proposed methods also perform a rank-one update at each iteration using one row of $C$ and the entire matrix $B$.
As summarized in Table~\ref{complexity}, the methods in \cite{ref11} and this paper require quadratic costs in terms of the matrix dimensions $m, n, p, q$, while the GI method in \cite{ref8} incurs a cubic complexity; see also Section~\ref{section6}.

We have first studied a (deterministic, cyclic) BK method in Section~\ref{section3.1} (Algorithm~\ref{ME-BK}).
This is a straightforward extension of the standard Kaczmarz method to matrix equation (\ref{eq1.1}) and also a deterministic (cyclic) version of the RBK method of~\cite{ref12}.
We have shown that this method converges to the solution $X^0_{\ast}$ $(=A^+CB^++X^0-A^+AX^0BB^+)$ if $0<\alpha<2\,\|B\|^{-2}$.
We have also considered variants of this BK method under the assumption that $B$ is of full column rank (Section~\ref{section3.2}, Algorithm~\ref{ME-BK1}) or full row rank (Section~\ref{section3.21}, Algorithm~\ref{alg.AX=C}).
These additional constraints enable us to derive a matrix expression for a sweep (cycle) of the BK method.
Note that Algorithm~\ref{ME-BK} imposes no requirements on $B$, but the last two methods require $B$ to be of full row or column rank (with nonsingularity as a special case) and of modest size.
In addition, as shown in Proposition~\ref{Alg._3_and_4_are_equivalent}, the last two methods are equivalent when the coefficient matrix $B$ is nonsingular.
In Experiment~\ref{exp6.2}, we have confirmed that these three BK methods exhibit a similar convergence behavior when $B$ is $3 \times 3$.

Next, we have proposed and studied the GRBK method (Algorithm~\ref{ME-GRBK}) and its relaxed and deterministic variants: RGRBK (Algorithm~\ref{ME-RGRBK}) and MWRBK (Algorithm~\ref{ME-MWRBK}) in Section~\ref{section4}, which select rows based on the residual norm.
Based on the experiments in Section~\ref{section6}, the extra cost of computing this residual norm is worthwhile; in addition, the residual can also serve as a suitable stopping criterion.
We have proven that they converge to the unique minimum norm solution $A^+CB^+$ if $0<\alpha<2\,\|B\|^{-2}$.
According to the convergence analysis in Sections~\ref{section3} and \ref{section4}, the upper bounds of the convergence factors for the proposed methods are minimized when $\alpha = \|B\|^{-2}$.
As shown in Section~\ref{section4.3}, the upper bound of the convergence factor for the GRBK, RGRBK, and MWRBK methods is not larger than that of RBK. Thus, it is reasonable to expect that they may require fewer iterations to achieve convergence.
The experiments in Sections~\ref{section5} and~\ref{section6} show that this is indeed the case.
The RGRBK method reduces to the GRBK method when $\theta = \tfrac12$ and to the MWRBK method when $\theta = 1$, as illustrated in Fig.~\ref{deltafig} and Table~\ref{index}.
Its convergence factor upper bound is decreasing with $\theta$. Therefore, the upper bound for the convergence factor is minimal at $\theta = 1$, where RGRBK coincides with MWRBK.

Finally, based on the theoretical analysis and experimental results, the overall performance of the proposed methods can be summarized as follows.
The BK method converges asymptotically (see Theorem~\ref{theorem3.3}); its convergence rate may be sensitive to the ordering of the rows in $A$. It may be (much) faster than its randomized counterpart (RBK), but it may also be slower. With a randomized initial ordering of the rows, the BK method typically behaves similarly to the randomized method RBK, without the drawback of the variance of RBK.
We note although its convergence is often slower than other methods (see Section~\ref{section6}), BK is valuable for understanding the influence of row selection through comparisons with the randomized (RBK) and greedy (GRBK, RGRBK, MWRBK) counterparts.
The greedy extensions (GRBK, RGRBK, and MWRBK) consistently achieve faster convergence than RBK, also observed in the numerical tests in Sections~\ref{section5} and \ref{section6}.
Although the greedy approaches are slower than the BK and RBK methods per iteration, they use fewer iterations with a lower runtime.
The randomized algorithms (RBK, GRBK, and RGRBK) have modest standard deviations in the experiments.
Overall, both theory and experiments suggest that the (deterministic) MWRBK method (Algorithm~\ref{ME-MWRBK}) seems to be the most attractive among the considered methods in terms of the number of iterations and computational time.
Finally, we remark that in this paper we have considered matrix equations without noise. Kaczmarz type methods may also be particularly suitable for noisy right-hand sides $C$; this is left for follow-up work.

\bmhead{Acknowledgments}
We are very grateful to three expert referees for their valuable suggestions.
The work is supported by the Fundamental Research Funds for the Central Universities (No.~2024YJS111) and a China Scholarship Council grant (No.~202407090119).

\bmhead{Conflict of interest} Not applicable.

\end{document}